\documentclass[reqno,tbtags]{amsproc}
\usepackage{amssymb}
\usepackage{epsfig}
\usepackage{chicago}
\usepackage{euscript}
\usepackage{eufrak}
\newtheorem{theorem}{Theorem}[section]
\newtheorem{corollary}{Corollary}[section]
\newtheorem{lemma}{Lemma}[section]
\theoremstyle{definition}
\newtheorem{definition}{Definition}[section]
\newtheorem{remark}{Remark}[section]

\def\IME{\emph{In\-su\-ran\-ce: Ma\-the\-ma\-tics and Eco\-no\-mics\/}}
\def\SAJ{\emph{Scan\-di\-na\-vian Ac\-tua\-rial Journal\/}}

\def\LMJ{\emph{Lith. Math. J.\/}}

\def\TPA{\emph{Theory Probab. Appl.}}

\def\AAP{\emph{Ad\-van\-ces in Ap\-pli\-ed Pro\-ba\-bi\-li\-ty\/}}

\def\JASA{\emph{Journal of the American Statistical Association\/}}
\numberwithin{equation}{section}
\def\eqOK{=}

\newcommand{\A}{\mathcal{A}}
\newcommand{\B}{\mathcal{B}}
\newcommand{\Mi}{\EuScript{M}}
\newcommand{\BesselI}[1]{I_{#1}}
\newcommand{\BesselK}[1]{K_{#1}}
\newcommand{\muIG}{\mu}
\newcommand{\lambdaIG}{\lambda}
\newcommand{\Int}[2]{{\mathcal{#2}}_{#1}}
\def\mmu{\hat{\mu}}
\newcommand{\OneVar}{\mathcal{X}}
\newcommand{\TwoVar}{\mathcal{V}}
\newcommand{\MainApprox}[1]{\mathcal{A}_{#1}}
\newcommand{\expaN}[1]{\mathcal{E}_{#1}}
\newcommand{\CoR}[1]{\mathcal{B}_{#1}}
\newcommand{\Elem}[2]{\mathcal{I}^{[#2]}_{#1}}
\newcommand{\KonstF}{C_{\mathcal{F}}\,}
\newcommand{\KonstS}{C_{\mathcal{S}}\,}
\newcommand{\MainRemTerm}[1]{\mathcal{R}_{#1}}
\def\cS{c^{*}}
\newcommand{\probR}[1]{\boldsymbol{\psi}_{#1}}
\def\paramT{\lambda}
\def\paramY{\mu}
\newcommand{\timeR}{\Upsilon}
\renewcommand{\P}{\mathsf{P}}
\newcommand{\p}{\mathsf{p}}
\newcommand{\R}{\mathsf{R}}
\newcommand{\D}{\mathsf{D}}
\newcommand{\E}{\mathsf{E}}
\newcommand{\EnOne}{N_{\epsilon}}
\newcommand{\Y}[1]{Y_{#1}}
\newcommand{\T}[1]{T_{#1}}
\newcommand{\UGauss}[2]{\varPhi_{\left({#1},{#2}\right)}}%{\Phi_{({#1},{#2})}}
\newcommand{\Ugauss}[2]{\varphi_{\left({#1},{#2}\right)}}%{\phi_{({#1},{#2})}}
\newcommand{\homN}[1]{N_{#1}}
\newcommand{\homV}[1]{V_{#1}}
\newcommand{\homR}[1]{R_{#1}}

\begin{document}
%%%%%%%%%%%%%%%%%%%%%%%%%%%%%%%%%%%%%%%%%%%%%%%%%%%%%%%%%%%%%%%%%%%%%%%%%%%%%%%%%%%%%%%%
\author[Vsevolod K. Malinovskii]{\Large Vsevolod K. Malinovskii}


\address{Central Economics and Mathematics Institute (CEMI) of Russian Academy of Science,
117418, Nakhimovskiy prosp., 47, Moscow, Russia}

\email{malinov@orc.ru, malinov@mi.ras.ru}

\urladdr{http://www.actlab.ru}

\title[GENERALIZED INVERSE GAUSSIAN DISTRIBUTIONS AND FIRST LEVEL CROSSING]{GENERALIZED
INVERSE GAUSSIAN DISTRIBUTIONS AND THE TIME OF FIRST LEVEL CROSSING}

\maketitle

\begin{abstract}
We propose a new approximation for the distribution of the time of the first
crossing of a high level $u$ by random process $\homV{s}-cs$, where $\homV{s}$,
$s>0$, is compound renewal process and $c>0$. It significantly outperforms the
existing approximations, particularly in the region around the critical point
$c=\cS$ which separates processes with positive and negative drifts. This
approximation is tightly related to generalized inverse Gaussian distributions.
\end{abstract}

%%%%%%%%%%%%%%%%%%%%%%%%%%%%%%%%%%%%%%%%%%%%%%%%%%%%%%%%%%%%%%%%%%%%%%%%%%%%%%%%%%%%%%%%
\section{Introduction}\label{sdrghrtjkf}
%%%%%%%%%%%%%%%%%%%%%%%%%%%%%%%%%%%%%%%%%%%%%%%%%%%%%%%%%%%%%%%%%%%%%%%%%%%%%%%%%%%%%%%%

Inverse Gaussian distribution (see \citeNP{[Jorgensen 1982]}, \citeNP{[Chhikara
1989]}, \citeNP{[Seshadri 1999]}) has probability density function (p.d.f.)
\begin{equation}\label{wqrdtgrehr}
f\big(x;\muIG,\lambdaIG,-\tfrac{1}{2}\big)=
\frac{\lambdaIG^{1/2}}{\sqrt{2\pi}}\,x^{-3/2}\exp\Big\{-\frac{\lambdaIG(x-\muIG)^2}{2\muIG^2
x}\Big\},
\end{equation}
where $x$, $\lambdaIG$, and $\muIG$ are positive\footnote{Parameter
$\lambdaIG>0$ is called shape parameter, and $\muIG>0$ is called mean
parameter.}. It is ``inverse'' in that sense that while Gaussian distribution
describes a Brownian motion's position at a fixed time, the inverse Gaussian
distribution describes the distribution of the time a Brownian motion with
positive drift takes to reach a fixed positive level.

Inverse Gaussian distribution has attracted a lot of researchers' interest.
\citeNP{[Seshadri 1997]} (see also \citeNP{[Morlat 1956]}) attributes its
invention to \citeNP{[Halpen 1941]}. Furthermore, \citeNP{[Chaudry 2002]}, with
reference to \citeNP{[Jorgensen 1982]}, attribute the invention of generalized
inverse Gaussian distribution to \citeNP{[Good 1953]}.

In the study of this distribution, paramount is finding explicit expression
\begin{multline*}
F\big(x;\muIG,\lambdaIG,-\tfrac{1}{2}\big)=\int_{0}^{x}f\big(z;\muIG,\lambdaIG,-\tfrac{1}{2}\big)dz
\\
\eqOK\UGauss{0}{1}\bigg(\sqrt{\frac{\lambdaIG}{x}}\bigg(\frac{x}{\muIG}-1\bigg)\bigg)
+\exp\bigg\{\frac{2\lambdaIG}{\muIG}\bigg\}\,\UGauss{0}{1}
\bigg(-\sqrt{\frac{\lambdaIG}{x}}\bigg(\frac{x}{\muIG}+1\bigg)\bigg)
\end{multline*}
for cumulative distribution function (c.d.f.) corresponding to p.d.f.
\eqref{wqrdtgrehr}; by $\UGauss{0}{1}(x)$ we denote c.d.f. of a standard normal
distribution\footnote{In Section 2.5 of the book \citeNP{[Chhikara 1989]}, the
authors say that \citeNP{[Shuster 1968]} expressed the cumulative distribution
function of the inverse Gaussian distribution in terms of a standard normal
distribution function, and that his proof is fairly complex and tedious. They
mention also \citeNP{[Zigangirov 1962]}. They give their own, rather
artificial, proof published in \citeNP{[Chhikara 1974]}.}. It seems that,
without pronouncing its present name, inverse Gaussian distribution was just
studied in \citeNP{[Binet 1841]}: this work is devoted to calculation of the
integrals like $\int_{0}^{x}f\big(z;\muIG,\lambdaIG,-\tfrac{1}{2}\big)dz$. The
same priority remark, as it seems, is applicable to a series of works devoted
to generalized incomplete Gamma function (see, e.g., \citeNP{[Chaudry 2001]},
\citeNP{[Chaudry 2002]}).

In this paper, instead of Brownian motion, we are focused on the random process
$\homV{s}-cs$, where $\homN{s}=\max\left\{n>0:\sum_{i=1}^{n}\T{i}\leqslant
s\right\}$, or $0$, if $\T{1}>s$, is renewal, and
$\homV{s}=\sum_{i=1}^{\homN{s}}\Y{i}$ or $0$, if $\homN{s}=0$ (or $\T{1}>s$),
is compound renewal processes. In risk theory, $\homV{s}$ and $\homN{s}$ are
called aggregate claim payout and claim arrival processes respectively. This
setting is important in various other fields of applied probability (see, e.g.,
\S~22 in \citeNP{[Takacs 1967]} for random walks with random displacements).

Put $\homR{s}=u+cs-\homV{s}$, $s\geqslant 0$. In risk theory, it is called risk
reserve process. We will show that inverse Gaussian and generalized inverse
Gaussian distributions play a paramount role in approximating
$\P\{\timeR\leqslant t\}$, where $\timeR=\inf\left\{s>0:\homR{s}<0\right\}$, or
$+\infty$, as $\homR{s}\geqslant 0$ for all $s>0$. It is the time of first
crossing of level $u$ by the process $\homV{s}-cs$. In risk theory, $\timeR$ is
called time of the first ruin, and $\P\{\timeR\leqslant
t\}=\probR{t}(u,c)=\P\left\{\inf_{0<s\leqslant t}\homR{s}<0\right\}$ is called
probability of ruin within time $t$.

Using associated random walks\footnote{In the random walk or risk theoretic
context, to pass to the associated random walk is considered basic technique
originating with Cram{\'e}er (see, e.g., Feller (1971), ch.~ XII, \S~4).} and
ladder technique, the approximations of $\probR{t}(u,c)$, as $u\to\infty$, were
investigated in \citeNP{[von Bahr 1974]} and in \citeNP{[Malinovskii 1994]}. In
\citeNP{[Malinovskii 2000]}, it was shown that this technique has limited
applicability for $c$ approaching $\cS$, as $u\to\infty$. The reasons for it,
deeply connected with the essence of this technique, were discussed in
\citeNP{[Malinovskii Kosova 2014]}. The present paper is a development of
\citeNP{[Malinovskii 2017]}, where more detailed discussion of the novelty of
our method is held, and more references are given.

%%%%%%%%%%%%%%%%%%%%%%%%%%%%%%%%%%%%%%%%%%%%%%%%%%%%%%%%%%%%%%%%%%%%%%%%%%%%%%%%%%%%%%%%
\section{Approximation for distribution of the time of first level crossing}\label{wdrtfhmn}
%%%%%%%%%%%%%%%%%%%%%%%%%%%%%%%%%%%%%%%%%%%%%%%%%%%%%%%%%%%%%%%%%%%%%%%%%%%%%%%%%%%%%%%%

Further in this paper, by $f_{\T{1}}(t)$, $f_{T}(t)$ and $f_{Y}(t)$ we denote
p.d.f. of the distribution of first time interval $\T{1}$, i.e., time between
starting time zero and time of the first event, of subsequent time intervals
$\T{i}\overset{d}{=}T$, $i=2,3,\dots$, and of jump sizes
$\Y{i}\overset{d}{=}Y$, $i=1,2,\dots$.
%From the standpoint of renewal process, it imposes certain well known restrictions on the model.
Being within renewal model, all these random variables are assumed mutually
independent.

Denote by $\P\{v<\timeR\leqslant t\mid\T{1}=v\}$ the distribution of $\timeR$
conditioned by $\T{1}=v$. It is easily seen that for $0<v<t$
\begin{equation}\label{ewrktulertye}
\P\{\timeR\leqslant
t\}=\int_{0}^{t}\P\{u+cv-\Y{1}<0\}f_{\T{1}}(v)dv+\int_{0}^{t}\P\{v<\timeR\leqslant
t\mid\T{1}=v\}f_{\T{1}}(v)dv.
\end{equation}
Put $M={\E{T}}/{\E{Y}}$, $D^2=((\E{T})^2\D{Y}+(\E{Y})^2\D{T})/(\E{Y})^3$, write
$\Ugauss{m}{s^2}$ for p.d.f. of a normal distribution with mean $m$ and
variance $s^2$, and introduce
\begin{equation}\label{wsdrthyjkr}
\expaN{t}(u,c,v)=\Int{t}{M}(u,c,v)+\KonstF\Int{t}{F}(u,c,v)+\KonstS\Int{t}{S}(u,c,v),
\end{equation}
where
\begin{equation}\label{wqerthrrthj}
\begin{aligned}
\Int{t}{M}(u,c,v)&=\int_{0}^{\frac{c(t-v)}{u+cv}}\frac{1}{1+x}
\,\Ugauss{cM(1+x)}{\frac{c^2D^2(1+x)}{u+cv}}(x)dx,
\\
\Int{t}{F}(u,c,v)&=\int_{0}^{\frac{c(t-v)}{u+cv}}\frac{x-Mc(1+x)}{(1+x)^2}
\Ugauss{cM(1+x)}{\frac{c^2D^2(1+x)}{u+cv}}(x)dx,
\\
\Int{t}{S}(u,c,v)&=\frac{u+cv}{c^2D^2}\int_{0}^{\frac{c(t-v)}{u+cv}}\frac{(x-Mc(1+x))^3}{(1+x)^3}
\Ugauss{cM(1+x)}{\frac{c^2D^2(1+x)}{u+cv}}(x)dx
\end{aligned}
\end{equation}
and\footnote{Here $\D{Y}=\E(Y-\E{Y})^2$, $\D{T}=\E(T-\E{T})^2$.}
\begin{equation*}
\begin{aligned}
\KonstF&\eqOK\frac{\E(T-\E{T})^3}{2cD^2\D{T}}\bigg(\dfrac{(\E{T})^2\D{Y}}{D^2(\E{Y})^3}-1\bigg)-\frac{\E{T}\E(Y-\E{Y})^3}{2cD^2\E{Y}\D{Y}}
\bigg(\dfrac{\D{T}}{D^2\E{Y}}-1\bigg)+\frac{\E{T}}{2cD^2},
\\[8pt]
\KonstS&\eqOK\frac{\E(T-\E{T})^3}{6cD^4\E{Y}}
-\dfrac{(\E{T})^3\E(Y-\E{Y})^3}{6cD^4(\E{Y})^4}
+\frac{\E{T}\D{Y}}{2cD^2(\E{Y})^2}.
\end{aligned}
\end{equation*}

\begin{figure}[t]
\includegraphics[scale=0.8]{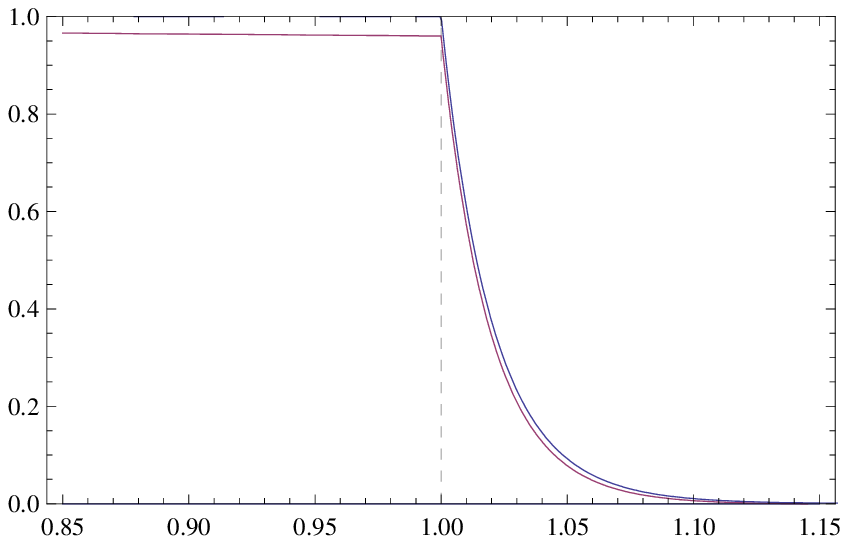}
\includegraphics[scale=0.8]{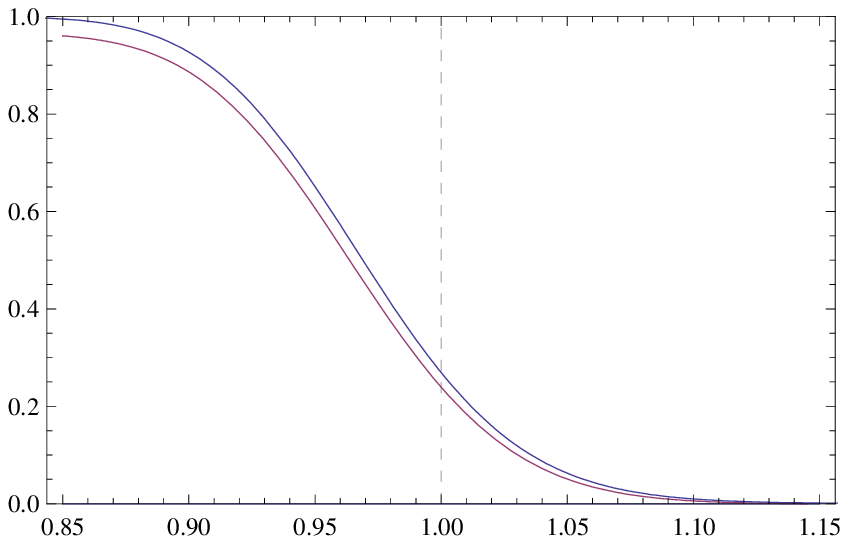}
\caption{\small Graphs ($X$-axis is $c$) of the functions
$\P\{v<\timeR\leqslant t\mid \T{1}=v\}$ (blue) given in Theorem~\ref{dthyjrt}
and $\expaN{t}(u,c,v)$ (red) defined in equation \eqref{wsdrthyjkr}, for
exponential $T$ with parameter $\paramT$ and exponential $Y$ with parameter
$\paramY$, as $\paramT=\paramY=1$, $v=0$, $u=50$, $t=\infty$ (above), $t=1000$
(below).}\label{ertgerhrh}
\end{figure}

\begin{theorem}\label{srdthjrf}
In the above model, let p.d.f. $f_{T}(y)$ and $f_{Y}(y)$ be bounded from above
by a finite constant, $D^2>0$, $\E({T}^4)<\infty$, $\E({Y}^4)<\infty$. Then for
$c>0$, for fixed $0<v<t$ we have
\begin{equation}\label{werhrjet}
\sup_{t>v}\Big|\,\P\{v<\timeR\leqslant t\mid\T{1}=v\}-\expaN{t}(u,c,v)\Big|
=\underline{O}\bigg(\frac{\ln(u+cv)}{(u+cv)^2}\bigg),
\end{equation}
as $u+cv\to\infty$.
\end{theorem}

Denote by\footnote{See e.g. \citeNP{[Abramowitz Stegun 1972]}, or
\citeNP{[Watson 1945]}, or Chapter XVII, Section 17.7 in \citeNP{[Whittaker
Watson 1963]}.} $\BesselI{1}(z)$ the modified Bessel function of the first kind
of order $1$.

\begin{theorem}\label{dthyjrt}
Assuming that $T$ and $Y$ are exponential with parameters $\paramT>0$ and
$\paramY>0$ respectively, for $0<v<t$ we have
\begin{multline*}
\P\{v<\timeR\leqslant t\mid \T{1}=v\}=\sqrt{\paramY\paramT
c}\,(v+u/c)e^{-\paramY u}e^{-\paramY cv}
\\
\times\int_{0}^{t-v}\frac{\BesselI{1}(2\sqrt{\paramY\paramT
c(y+v+u/c)y})}{\sqrt{(y+v+u/c)y}} e^{-(\paramY c+\paramT)y}dy.
\end{multline*}
\end{theorem}

In Section~\ref{rthytrjrt}, we will show that
\begin{equation*}
\Int{t}{M}(u,c,v)=\underline{O}(1),\quad
\Int{t}{F}(u,c,v)=\underline{O}((u+cv)^{-1}),\quad
\Int{t}{S}(u,c,v)=\underline{O}((u+cv)^{-1}),
\end{equation*}
as $u+cv\to\infty$, with explicitly written right-hand sides. The latter means
that $\Int{t}{M}(u,c,v)$, $\Int{t}{F}(u,c,v)$, $\Int{t}{S}(u,c,v)$ in
\eqref{wsdrthyjkr} will be expressed in terms of c.d.f. of generalized inverse
Gaussian distributions. This converts the approximation \eqref{werhrjet} into
the usual-form asymptotic expansions with explicitly written main and first
correction terms.

%%%%%%%%%%%%%%%%%%%%%%%%%%%%%%%%%%%%%%%%%%%%%%%%%%%%%%%%%%%%%%%%%%%%%%%%%%%%%%%%%%%%%%%%
\section{Elementary components and generalized inverse Gaussian distributions}\label{sdrthrjk}
%%%%%%%%%%%%%%%%%%%%%%%%%%%%%%%%%%%%%%%%%%%%%%%%%%%%%%%%%%%%%%%%%%%%%%%%%%%%%%%%%%%%%%%%

\begin{definition}\label{rterherher}
By elementary\footnote{Elementary, as compared to $\Int{t}{M}(u,c,v)$,
$\Int{t}{F}(u,c,v)$, $\Int{t}{S}(u,c,v)$ introduced in \eqref{wqerthrrthj}.}
components we call
\begin{equation*}
\Elem{t}{k}(u,c,v)=\int_{0}^{\frac{c(t-v)}{u+cv}}\frac{1}{(1+x)^{k}}\,
\Ugauss{cM(1+x)}{\frac{c^2D^2(1+x)}{u+cv}}(x)\,dx,
\end{equation*}
where $k=0,1,2,\dots$.
\end{definition}

By $\BesselK{p}(z)$ we denote modified Bessel function of the second kind (see
Section~\ref{sdrthjrnm}).

\begin{definition}\label{sadfgerhej}
The generalized inverse Gaussian distribution with real $p$ and $\lambdaIG>0$,
$\muIG>0$ is given by p.d.f.
\begin{multline}\label{45t34y34}
f(x;\muIG,\lambdaIG,p)=\frac{1}{2\muIG^p\BesselK{p}(\frac{\lambdaIG}{\muIG})}\,
x^{p-1}\exp\Big\{-\frac{\lambdaIG(x^2+\muIG^2)}{2\muIG^2x}\Big\}
\\
=\frac{e^{-\frac{\lambdaIG}{\muIG}}}{2\muIG^p\BesselK{p}(\frac{\lambdaIG}{\muIG})}\,x^{p-1}
\exp\Big\{-\frac{\lambdaIG(x-\muIG)^2}{2\muIG^2 x}\Big\},\quad x>0.
\end{multline}
\end{definition}

Bearing in mind that\footnote{See Lemma~\ref{sdrjrnm} which yields the
expressions for $\BesselK{1/2}(z)=\BesselK{-1/2}(z)$, $\BesselK{3/2}(z)$, and
$\BesselK{5/2}(z)$.}
$\BesselK{1/2}(z)=\BesselK{-1/2}(z)=\frac{\sqrt{\pi}}{\sqrt{2z}}e^{-z}$, for
$p=\frac{1}{2}$ the equality \eqref{45t34y34} rewrites as
\begin{equation}\label{srdth}
f\big(x;\muIG,\lambdaIG,\tfrac{1}{2}\big)
=\frac{\lambdaIG^{1/2}}{\muIG\sqrt{2\pi}}\,x^{-1/2}\exp\Big\{-\frac{\lambdaIG(x-\muIG)^2}{2\muIG^2
x}\Big\},\quad x>0.
\end{equation}
For $p=-\frac{1}{2}$, the equality \eqref{45t34y34} rewrites as
\begin{equation}\label{rtuyjtjty}
f\big(x;\muIG,\lambdaIG,-\tfrac{1}{2}\big)=
\frac{\lambdaIG^{1/2}}{\sqrt{2\pi}}\,x^{-3/2}\exp\Big\{-\frac{\lambdaIG(x-\muIG)^2}{2\muIG^2
x}\Big\},\quad x>0.
\end{equation}
Plainly (cf. \eqref{wqrdtgrehr}), this is  p.d.f. of a standard inverse
Gaussian distribution.

Bearing in mind that
$\BesselK{3/2}(z)\eqOK\frac{\sqrt{\pi}}{\sqrt{2z}}e^{-z}(1+z^{-1})$, for
$p=-\frac{3}{2}$ the equality \eqref{45t34y34} rewrites as
\begin{equation}\label{gyhjjryjk}
f\big(x;\muIG,\lambdaIG,-\tfrac{3}{2}\big)
\eqOK\frac{\lambdaIG^{3/2}\muIG}{\sqrt{2\pi}(\lambdaIG+\muIG)}\,x^{-5/2}\exp\Big\{-\frac{\lambdaIG(x-\muIG)^2}{2\muIG^2
x}\Big\},\quad x>0.
\end{equation}

Bearing in mind that
$\BesselK{5/2}(z)\eqOK\frac{\sqrt{\pi}}{\sqrt{2z}}e^{-z}(1+3z^{-1}+3z^{-2})$,
for $p=-\frac{5}{2}$ the equality \eqref{45t34y34} rewrites as
\begin{equation}\label{retyurtikt7}
f\big(x;\muIG,\lambdaIG,-\tfrac{5}{2}\big)=\frac{\lambdaIG^{5/2}\muIG^{2}}{{\sqrt{2\pi}}\big(\lambdaIG^2+3\lambdaIG
\muIG+3\muIG^2\big)}\,x^{-7/2}\exp\Big\{-\frac{\lambdaIG(x-\muIG)^2}{2\muIG^2
x}\Big\},\quad x>0.
\end{equation}

In Section~\ref{rtgwerghwer}, we outlined the method by \citeNP{[Binet 1841]}.
It allows us to calculate c.d.f. corresponding to p.d.f.
\eqref{srdth}--\eqref{retyurtikt7} in an explicit form. For brevity, we skip
detailed demonstration of this calculation. The reader can verify its
correctness by means of direct differentiation of c.d.f. given below in
Theorems \ref{sdrwhgtrjn}--\ref{wdrthytm}.

\begin{theorem}\label{sdrwhgtrjn}
For $\lambdaIG>0$, $\muIG>0$, we have
\begin{multline*}
F\big(x;\muIG,\lambdaIG,\tfrac{1}{2}\big)=\int_{0}^{x}f\big(z;\muIG,\lambdaIG,\tfrac{1}{2}\big)dz
\\
\eqOK\UGauss{0}{1}\bigg(\sqrt{\frac{\lambdaIG}{x}}
\Big(\frac{x}{\muIG}-1\Big)\bigg)-\exp\Big\{\frac{2\lambdaIG}{\muIG}\Big\}
\UGauss{0}{1}\bigg(-\sqrt{\frac{\lambdaIG}{x}}
\Big(1+\frac{x}{\muIG}\Big)\bigg),\quad x>0.
\end{multline*}
\end{theorem}

\begin{theorem}\label{sdtyhjtm}
For $\lambdaIG>0$ and $\muIG>0$, we have
\begin{multline*}
F\big(x;\muIG,\lambdaIG,-\tfrac{1}{2}\big)=\int_{0}^{x}f\big(z;\muIG,\lambdaIG,-\tfrac{1}{2}\big)dz
\\
\eqOK\UGauss{0}{1}\bigg(\sqrt{\frac{\lambdaIG}{x}}\bigg(\frac{x}{\muIG}-1\bigg)\bigg)
+\exp\bigg\{\frac{2\lambdaIG}{\muIG}\bigg\}\,\UGauss{0}{1}
\bigg(-\sqrt{\frac{\lambdaIG}{x}}\bigg(\frac{x}{\muIG}+1\bigg)\bigg),\quad x>0.
\end{multline*}
\end{theorem}

\begin{theorem}\label{styhjrtjf}
For $\lambdaIG>0$, $\muIG>0$, we have
\begin{multline*}
F\big(x;\muIG,\lambdaIG,-\tfrac{3}{2}\big)=\int_{0}^{x}f\big(z;\muIG,\lambdaIG,-\tfrac{3}{2}\big)dz
\\
\eqOK\UGauss{0}{1}\bigg(\sqrt{\frac{\lambdaIG}{x}}\bigg(\frac{x}{\muIG}-1\bigg)\bigg)
-\frac{\lambdaIG-\muIG}{\lambdaIG+\muIG}\exp\bigg\{\frac{2\lambdaIG}{\muIG}\bigg\}
\UGauss{0}{1}\bigg(-\sqrt{\frac{\lambdaIG}{x}}\bigg(\frac{x}{\muIG}+1\bigg)\bigg)
\\
+\frac{\sqrt{2\lambdaIG}\,\muIG}{\sqrt{\pi x}(\lambdaIG+\muIG)}
\exp\bigg\{\frac{\lambdaIG}{\muIG}\bigg\}\exp\bigg\{-\frac{\lambdaIG}{2x}\bigg(\frac{x^2}{\muIG^2}+1\bigg)\bigg\},\quad
x>0.
\end{multline*}
\end{theorem}

\begin{theorem}\label{wdrthytm}
For $\lambdaIG>0$, $\muIG>0$, we have
\begin{multline*}
F\big(x;\muIG,\lambdaIG,-\tfrac{5}{2}\big)
=\int_{0}^{x}f\big(z;\muIG,\lambdaIG,-\tfrac{5}{2}\big)dz
\\
\eqOK\UGauss{0}{1}\Bigg(\sqrt{\frac{\lambdaIG}{x}}\bigg(\frac{x}{\muIG}-1\bigg)\Bigg)
+\frac{\lambdaIG^2-3\lambdaIG\muIG+3\muIG^2}{\lambdaIG^2+3\lambdaIG\muIG+3\muIG^2}
\exp\bigg\{\frac{2\lambdaIG}{\muIG}\bigg\}
\UGauss{0}{1}\Bigg(-\sqrt{\frac{\lambdaIG}{x}}\bigg(\frac{x}{\muIG}+1\bigg)\Bigg)
\\
+\frac{\sqrt{2\lambdaIG}\,\muIG^2(\lambdaIG+3x)}{\sqrt{\pi}\big(\lambdaIG^2+3\lambdaIG\muIG+3\muIG^2\big)
x^{3/2}}
\exp\bigg\{\frac{\lambdaIG}{\muIG}\bigg\}\exp\bigg\{-\frac{\lambdaIG}{2x}\bigg(\frac{x^2}{\muIG^2}+1\bigg)\bigg\},\quad
x>0.
\end{multline*}
\end{theorem}

Let us express the elementary components $\Elem{t}{k}(u,c,v)$, $k=0,1,2,3$,
first through c.d.f. of generalized inverse Gaussian distributions, and second
through $\UGauss{0}{1}(x)$.

\begin{theorem}\label{sdfghreh}
For $\cS=\frac{1}{M}$, $\lambdaIG=\frac{u+cv}{c^2D^2}>0$,
$\muIG=\frac{1}{1-cM}$, and $\mmu=-\muIG=\frac{1}{cM-1}$, we have
\begin{equation}\label{etfyhjkrt}
\Elem{t}{0}(u,c,v)=\begin{cases} \muIG
F\big(x;\muIG,\lambdaIG,\tfrac{1}{2}\big)\Big|_{x=1}^{\frac{c(t-v)}{u+cv}+1},&0<c\leqslant\cS,
\\[6pt]
\mmu\exp\big\{-2\,\tfrac{\lambdaIG}{\mmu}\big\}
F\big(x;\mmu,\lambdaIG,\tfrac{1}{2}\big)\Big|_{x=1}^{\frac{c(t-v)}{u+cv}+1},&
c\geqslant\cS.
\end{cases}
\end{equation}
\end{theorem}

\begin{proof}
Note that
\begin{multline}\label{ertehtrjt}
\Ugauss{cM(1+x)}{\frac{c^2D^2(1+x)}{u+cv}}(x)
=\frac{1}{\sqrt{2\pi}}\exp\bigg\{\frac{\lambdaIG}{\muIG}\bigg\}
\exp\bigg\{-\frac{\lambdaIG((1+x)^2+\muIG^2)}{2\muIG^2 (1+x)}\bigg\}
\\
=\frac{1}{\sqrt{2\pi}}\exp\bigg\{-\frac{\lambdaIG}{\mmu}\bigg\}
\exp\bigg\{-\frac{\lambdaIG((1+x)^2+\mmu^2)}{2\mmu^2 (1+x)}\bigg\}.
\end{multline}
For $0<c\leqslant\cS$, from where follows $\muIG>0$, we use the first equation
\eqref{ertehtrjt} and have
\begin{equation*}\label{sarfghfdn}
\begin{aligned}
\Elem{t}{0}(u,c,v)=\muIG\int_{0}^{\frac{c(t-v)}{u+cv}}
f\big(1+x;\muIG,\lambdaIG,\tfrac{1}{2}\big)dx.
\end{aligned}
\end{equation*}
For $c\geqslant\cS$, from where follows $\mmu>0$, we use the second equation
\eqref{ertehtrjt} and have
\begin{equation*}
\Elem{t}{0}(u,c,v)=\mmu\exp\big\{-2\,\tfrac{\lambdaIG}{\mmu}\big\}\int_{0}^{\frac{c(t-v)}{u+cv}}
f\big(1+x;\mmu,\lambdaIG,\tfrac{1}{2}\big)dx.
\end{equation*}
Applying Theorem~\ref{sdrwhgtrjn} to these integrals, we get \eqref{etfyhjkrt},
as required.
\end{proof}

Taking advantage of Theorem~\ref{sdrwhgtrjn}, we rewrite \eqref{etfyhjkrt} as
\begin{equation}\label{srdfghrtjhr}
\begin{aligned}
\Elem{t}{0}(u,c,v)=\begin{cases}
\muIG\bigg[\UGauss{0}{1}\Big(\sqrt{\frac{\lambdaIG}{x}}
\Big(\frac{x}{\muIG}-1\Big)\Big)
\\[-4pt]
-\exp\Big\{\frac{2\lambdaIG}{\muIG}\Big\}
\UGauss{0}{1}\Big(-\sqrt{\frac{\lambdaIG}{x}}
\Big(1+\frac{x}{\muIG}\Big)\Big)\bigg]\bigg|_{x=1}^{\frac{c(t-v)}{u+cv}+1},&0<c\leqslant\cS,
\\[12pt]
\mmu\exp\big\{-2\,\tfrac{\lambdaIG}{\mmu}\big\}
\bigg[\UGauss{0}{1}\bigg(\sqrt{\frac{\lambdaIG}{x}}
\Big(\frac{x}{\mmu}-1\Big)\bigg)
\\[-4pt]
-\exp\Big\{\frac{2\lambdaIG}{\mmu}\Big\}
\UGauss{0}{1}\bigg(-\sqrt{\frac{\lambdaIG}{x}}
\Big(1+\frac{x}{\mmu}\Big)\bigg)\bigg]\bigg|_{x=1}^{\frac{c(t-v)}{u+cv}+1},&
c\geqslant\cS,
\end{cases}
\end{aligned}
\end{equation}
which can be rewritten for all $c>0$ as
\begin{multline}\label{wqertgher}
\Elem{t}{0}(u,c,v)=\frac{1}{1-cM}\bigg[\UGauss{0}{1}\Big(\frac{\sqrt{u+cv}}{cD\sqrt{x}}
\big(x(1-cM)-1\big)\Big)
\\
-\exp\Big\{\frac{2(u+cv)}{c^2D^2}(1-cM)\Big\}
\UGauss{0}{1}\Big(-\frac{\sqrt{u+cv}}{cD\sqrt{x}}
\big(1+x(1-cM)\big)\Big)\bigg]\bigg|_{x=1}^{\frac{c(t-v)}{u+cv}+1}.
\end{multline}

\begin{theorem}\label{qwergtrh}
For $\cS=\frac{1}{M}$, $\lambdaIG=\frac{u+cv}{c^2D^2}>0$,
$\muIG=\frac{1}{1-cM}$, and $\mmu=-\muIG=\frac{1}{cM-1}$, we have
\begin{equation}\label{sdfghdtn}
\Elem{t}{1}(u,c,v)=\begin{cases}
F\big(x;\muIG,\lambdaIG,-\tfrac{1}{2}\big)\Big|_{x=1}^{\frac{c(t-v)}{u+cv}+1},&0<c\leqslant\cS,
\\[6pt]
\exp\big\{-2\,\tfrac{\lambdaIG}{\mmu}\big\}
F\big(x;\mmu,\lambdaIG,-\tfrac{1}{2}\big)\Big|_{x=1}^{\frac{c(t-v)}{u+cv}+1},&
c\geqslant\cS.
\end{cases}
\end{equation}
\end{theorem}

\begin{proof}
The proof is quite similar to the proof of Theorem~\ref{sdfghreh}.
\end{proof}

Taking advantage of Theorem~\ref{sdtyhjtm}, we rewrite \eqref{sdfghdtn} first
in the form similar to \eqref{srdfghrtjhr}, and thereafter, for all $c>0$, as
\begin{multline}\label{we5thrjt}
\Elem{t}{1}(u,c,v)=\bigg[\UGauss{0}{1}\Big(\frac{\sqrt{u+cv}}{cD\sqrt{x}}
\big(x(1-cM)-1\big)\Big)
\\
+\exp\Big\{\frac{2(u+cv)}{c^2D^2}(1-cM)\Big\}
\UGauss{0}{1}\Big(-\frac{\sqrt{u+cv}}{cD\sqrt{x}}
\big(1+x(1-cM)\big)\Big)\bigg]\bigg|_{x=1}^{\frac{c(t-v)}{u+cv}+1}.
\end{multline}

\begin{theorem}\label{gyjrtjr}
For $\cS=\frac{1}{M}$, $\lambdaIG=\frac{u+cv}{c^2D^2}>0$,
$\muIG=\frac{1}{1-cM}$, and $\mmu=-\muIG=\frac{1}{cM-1}$, we have
\begin{equation}\label{tryjujkr}
\Elem{t}{2}(u,c,v)=\begin{cases} \frac{(\lambdaIG+\muIG)}{\muIG\lambdaIG}
F\big(x;\muIG,\lambdaIG,-\tfrac{3}{2}\big)
\Big|_{x=1}^{\frac{c(t-v)}{u+cv}+1},&0<c\leqslant\cS,
\\[6pt]
\frac{(\lambdaIG+\mmu)}{\lambdaIG\mmu}\exp\big\{-2\frac{\lambdaIG}{\mmu}\big\}
F\big(x;\mmu,\lambdaIG,-\tfrac{3}{2}\big)\Big|_{x=1}^{\frac{c(t-v)}{u+cv}+1},&c\geqslant\cS.
\end{cases}
\end{equation}
\end{theorem}

\begin{proof}
The proof is quite similar to the proof of Theorem~\ref{sdfghreh}.
\end{proof}

Taking advantage of Theorem~\ref{styhjrtjf}, we rewrite \eqref{tryjujkr} in the
form similar to \eqref{srdfghrtjhr}, and thereafter, for all $c>0$, as
\begin{multline}\label{wqertrjnr}
\Elem{t}{2}(u,c,v)=\Big(1-cM+\frac{c^2D^2}{u+cv}\Big)
\bigg[\UGauss{0}{1}\Big(\frac{\sqrt{u+cv}}{cD\sqrt{x}}\big(x(1-cM)-1\big)\Big)
\\
-\frac{(u+cv)(1-cM)-c^2D^2}{(u+cv)(1-cM)+c^2D^2}
\exp\Big\{\frac{2(u+cv)}{c^2D^2}(1-cM)\Big\}
\UGauss{0}{1}\Big(-\frac{\sqrt{u+cv}}{cD\sqrt{x}}\big(1+x(1-cM)\big)\Big)
\\
+\frac{\sqrt{2(u+cv)}\,cD}{\sqrt{\pi x}((u+cv)(1-cM)+c^2D^2)}
\exp\Big\{\frac{(u+cv)}{c^2D^2}(1-cM)\Big\}
\\
\times\exp\bigg\{-\frac{u+cv}{2c^2D^2x}
\big(1+x^2(1-cM)^2\big)\bigg\}\bigg]\bigg|_{x=1}^{\frac{c(t-v)}{u+cv}+1}.
\end{multline}

\begin{theorem}\label{rtyhtrjr}
For $\cS=\frac{1}{M}$, $\lambdaIG=\frac{u+cv}{c^2D^2}>0$,
$\muIG=\frac{1}{1-cM}$, and $\mmu=-\muIG=\frac{1}{cM-1}$, we have
\begin{equation}\label{etryutk}
\Elem{t}{3}(u,c,v)=\begin{cases}
\frac{\lambdaIG^2+3\lambdaIG\muIG+3\muIG^2}{\muIG^2\lambdaIG^2}
F\big(x;\muIG,\lambdaIG,-\tfrac{5}{2}\big)
\Big|_{x=1}^{\frac{c(t-v)}{u+cv}+1},&0<c\leqslant\cS,
\\[6pt]
\frac{\lambdaIG^2+3\lambdaIG\mmu+3\mmu^2}{\mmu^2\lambdaIG^2}\exp\big\{-2\frac{\lambdaIG}{\mmu}\big\}
F\big(x;\mmu,\lambdaIG,-\tfrac{3}{2}\big)\Big|_{x=1}^{\frac{c(t-v)}{u+cv}+1},&c\geqslant\cS.
\end{cases}
\end{equation}
\end{theorem}

\begin{proof}
The proof is quite similar to the proof of Theorem~\ref{sdfghreh}.
\end{proof}

Taking advantage of Theorem~\ref{wdrthytm}, we rewrite \eqref{etryutk} in the
form similar to \eqref{srdfghrtjhr}, and thereafter, for all $c>0$, in the form
similar to \eqref{wqertgher}, \eqref{we5thrjt}, and \eqref{wqertrjnr}. We skip
this formula which is straightforward, but cumbersome.

%%%%%%%%%%%%%%%%%%%%%%%%%%%%%%%%%%%%%%%%%%%%%%%%%%%%%%%%%%%%%%%%%%%%%%%%%%%%%%%%%%%%%%%%
\section{Explicit expressions for $\expaN{t}(u,c,v)$}\label{rthytrjrt}
%%%%%%%%%%%%%%%%%%%%%%%%%%%%%%%%%%%%%%%%%%%%%%%%%%%%%%%%%%%%%%%%%%%%%%%%%%%%%%%%%%%%%%%%

%%%%%%%%%%%%%%%%%%%%%%%%%%%%%%%%%%%%%%%%%%%%%%%%%%%%%%%%%%%%%%%%%%%%%%%%%%%%%%%%%%%%%%%%
\subsection{Summand $\Int{t}{M}(u,c,v)=\int_{0}^{\frac{c(t-v)}{u+cv}}\frac{1}{1+x}
\,\Ugauss{cM(1+x)}{\frac{c^2D^2(1+x)}{u+cv}}(x)dx$}\label{tyjufktl}
%%%%%%%%%%%%%%%%%%%%%%%%%%%%%%%%%%%%%%%%%%%%%%%%%%%%%%%%%%%%%%%%%%%%%%%%%%%%%%%%%%%%%%%%

\begin{theorem}
For $x>0$, $c>0$, $u>0$, $t>0$, $0<v<t$, we have
\begin{equation*}
\begin{aligned}
\Int{t}{M}(u,c,v)&\eqOK\bigg[\UGauss{0}{1}\Big(\tfrac{\sqrt{u+cv}}{cD\sqrt{x}}\big(x(1-cM)-1\big)\Big)
\\[-4pt]
&\hskip
28pt+\exp\Big\{2\frac{u+cv}{c^2D^2}(1-cM)\Big\}\UGauss{0}{1}\Big(-\tfrac{\sqrt{u+cv}}{cD
\sqrt{x}}\big(x(1-cM)+1\big)\Big)\bigg]\bigg|_{x=1}^{\frac{c(t-v)}{u+cv}+1}.
\end{aligned}
\end{equation*}
\end{theorem}

\begin{proof}
Plainly, $\Int{t}{M}(u,c,v)$ defined in \eqref{wqerthrrthj} is equal to
$\Elem{t}{1}(u,c,v)$. It is just found in Theorem~\ref{qwergtrh}, and equation
that we have to prove is just written in \eqref{we5thrjt}.
\end{proof}

\begin{corollary}\label{wrgtreger}
We have
\begin{equation*}
\begin{aligned}
\Int{t}{M}(u,\cS,v)&\eqOK 2\bigg[1-\UGauss{0}{1}\Big(\tfrac{\sqrt{u+\cS v}}{\cS
D\sqrt{x}}\Big)\bigg]\bigg|_{x=1}^{\frac{\cS (t-v)}{u+\cS v}+1},
\end{aligned}
\end{equation*}
and $\Int{\infty}{M}(u,\cS,v)\eqOK 2\UGauss{0}{1}\Big(\tfrac{\sqrt{u+\cS
v}}{\cS D}\Big)-1$. Bearing in mind that
$\lim_{x\to\infty}e^{x^2/2}\UGauss{0}{1}(-x)=0$, we have
\begin{equation*}
\begin{aligned}
\Int{t}{M}(u,0,v)&\eqOK\UGauss{0}{1}\Big(\tfrac{t-Mu}{D\sqrt{u}}\Big)
-\UGauss{0}{1}\Big(-\tfrac{M}{D}\sqrt{u}\Big),
\end{aligned}
\end{equation*}
and $\Int{\infty}{M}(u,0,v)\eqOK\UGauss{0}{1}\Big(\tfrac{M}{D}\sqrt{u}\Big)$.
\end{corollary}

\begin{figure}[t]
\includegraphics[scale=0.8]{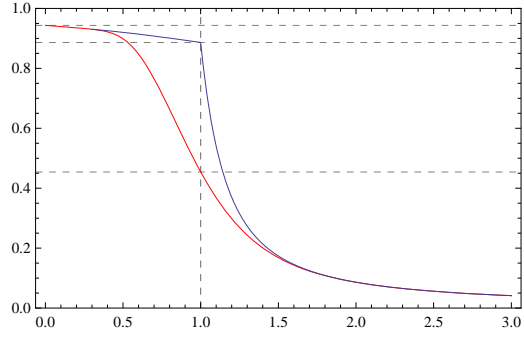}
\caption{\small Graphs ($X$-axis is $c$) of $\Int{t}{M}(u,c,v)$ with $t=100$,
and $\Int{t}{M}(u,c,v)$ with $t=\infty$. Here $v=0$, $u=15$, $M=1$, $D^2=6$.
Horizontal lines are $0.943$, $0.886$, and $0.454$.}\label{sdrtgerher}
\end{figure}

\begin{lemma}\label{srdthrjhrr}
For $0<c<\cS$, the function $\Int{\infty}{M}(u,c,v)$ is monotone decreasing, as
$c$ increases.
\end{lemma}

\begin{proof}
For $0<c<\cS$, let us show that $\frac{\partial}{\partial
c}\Int{\infty}{M}(u,c,v)<0$. For brevity, we confine ourselves with the case
$v=0$ and differentiate $\Int{\infty}{M}(u,c,0)$ straightforwardly. We have
\begin{multline*}
\frac{\partial}{\partial c}\Int{\infty}{M}(u,c,0)
=-\frac{2}{c^2D}\exp\bigg\{\frac{2(1-cM)u}{c^2D^2}\bigg\}\sqrt{u}\,
\Ugauss{0}{1}(\xi)\bigg(1-\xi\tfrac{\UGauss{0}{1}(-\xi)}
{\Ugauss{0}{1}(\xi)}\bigg)\bigg|_{\xi=\frac{2-cM}{cD}\sqrt{u}}.
\end{multline*}
Addressing to Mills' ratio $\Mi(\xi)=\tfrac{\UGauss{0}{1}(-\xi)}
{\Ugauss{0}{1}(\xi)}$, and bearing in mind that $1-\xi\Mi(\xi)>0$ for all
$\xi\in\R$, we get the required result.
\end{proof}

Taking advantage of Lemma~\ref{srdthrjhrr}, we observe that
$\Int{\infty}{M}(u,c,v)$ is sandwiched between
\begin{equation*}
\Int{\infty}{M}(u,0,v)=\UGauss{0}{1}\big(\tfrac{M}{D}\sqrt{u}\big)
\end{equation*}
and
\begin{equation*}
\Int{\infty}{M}(u,\cS,v)=2\UGauss{0}{1}\big(\tfrac{M}{D}\sqrt{u+\cS v}\big)-1
\end{equation*}
all over $0<c<\cS$. Plainly, $\UGauss{0}{1}\big(\frac{M}{D}\sqrt{u}\big)\to 1$
and $2\UGauss{0}{1}\big(\frac{M}{D}\sqrt{u+\cS v}\big)-1\to 1$, as
$u\to\infty$, and the function $\Int{\infty}{M}(u,c,v)$ approaches $1$
uniformly on $0<c<\cS$, as $u\to\infty$.

In Fig.~\ref{sdrtgerher}, we draw the functions $\Int{t}{M}(u,c,v)$ and
$\Int{\infty}{M}(u,c,v)$ for $v=0$, $M=1$, $D^2=6$, $t=100$, and $u=15$. The
former is smooth and monotone decreasing on the entire range of $c$, while the
latter is monotone decreasing, but has a nonsmoothness in the point $c=\cS$.
All over $0<c<\cS$, the function $\Int{\infty}{M}(u,c,v)$ is sandwiched between
$\Int{\infty}{M}(u,0,v)=0.943$ and $\Int{\infty}{M}(u,\cS,v)=0.886$ drawn by
dashed horizontal lines. The third dashed horizontal line is
$\Int{t}{M}(u,\cS,v)=0.454$.

%%%%%%%%%%%%%%%%%%%%%%%%%%%%%%%%%%%%%%%%%%%%%%%%%%%%%%%%%%%%%%%%%%%%%%%%%%%%%%%%%%%%%%%%
\subsection{Summand $\Int{t}{F}(u,c,v)=\int_{0}^{\frac{c(t-v)}{u+cv}}\frac{x-Mc(1+x)}{(1+x)^2}
\Ugauss{cM(1+x)}{\frac{c^2D^2(1+x)}{u+cv}}(x)dx$}\label{tyjrttrtrtuj}
%%%%%%%%%%%%%%%%%%%%%%%%%%%%%%%%%%%%%%%%%%%%%%%%%%%%%%%%%%%%%%%%%%%%%%%%%%%%%%%%%%%%%%%%

\begin{theorem}
For $c>0$, $u>0$, $t>0$, $0<v<t$, we have
\begin{equation*}
\begin{aligned}
\Int{t}{F}(u,c,v)&\eqOK-\frac{c^2D^2}{u+cv}\bigg[\UGauss{0}{1}\Big(\tfrac{\sqrt{u+cv}}{cD\sqrt{x}}
\big(x(1-cM)-1\big)\Big)
\\[-2pt]
&\hskip 32pt+\exp\Big\{2\frac{u+cv}{c^2D^2}(1-cM)\Big\}
\UGauss{0}{1}\Big(-\tfrac{\sqrt{u+cv}}{cD\sqrt{x}}\big(x(1-cM)+1\big)\Big)
\bigg]\bigg|_{x=1}^{\frac{c(t-v)}{u+cv}+1}
\\
&+2(1-cM)\exp\Big\{2\frac{u+cv}{c^2D^2}(1-cM)\Big\}\UGauss{0}{1}\Big(-\tfrac{\sqrt{u+cv}}{cD\sqrt{x}}
\big(x(1-cM)+1\big)\Big)\bigg|_{x=1}^{\frac{c(t-v)}{u+cv}+1}
\\
&-\frac{2cD}{\sqrt{2\pi x(u+cv)}}\exp\Big\{-\frac{u+cv}{2xc^2D^2}
\big(x(1-cM)-1\big)^2\Big\}\bigg|_{x=1}^{\frac{c(t-v)}{u+cv}+1}.
\end{aligned}
\end{equation*}
\end{theorem}

\begin{figure}[t]
\includegraphics[scale=0.8]{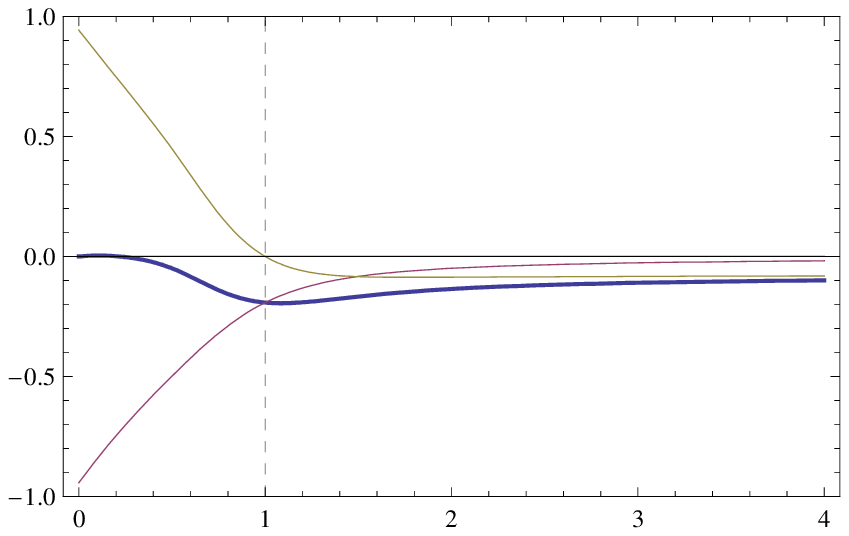}
\caption{\small Graphs ($X$-axis is $c$) of $\Int{t}{F}(u,c,v)$ and of
$(1-Mc)\Elem{t}{1}(u,c,v)$ and $-\Elem{t}{2}(u,c,v)$. Here $t=100$, $v=0$,
$u=15$, $M=1$, $D^2=6$.}\label{dfgjhkm}
\end{figure}

\begin{proof}
Observing that $\Int{t}{F}(u,c,v)=(1-cM)\Elem{t}{1}(u,c,v)-\Elem{t}{2}(u,c,v)$
and bearing in mind Theorems~\ref{qwergtrh} and~\ref{gyjrtjr}, we have
\begin{equation*}
\Int{t}{F}(u,c,v)=\begin{cases} \frac{1}{\muIG}
F\big(x;\muIG,\lambdaIG,-\tfrac{1}{2}\big)-\frac{(\lambdaIG+\muIG)}{\muIG\lambdaIG}
F\big(x;\muIG,\lambdaIG,-\tfrac{3}{2}\big)\Big|_{x=1}^{\frac{c(t-v)}{u+cv}+1},&0<c\leqslant\cS,
\\[12pt]
-\frac{1}{\mmu}\exp\big\{-2\,\tfrac{\lambdaIG}{\mmu}\big\}
F\big(x;\mmu,\lambdaIG,-\tfrac{1}{2}\big)
\\
-\frac{(\lambdaIG+\mmu)}{\lambdaIG\mmu}\exp\big\{-2\frac{\lambdaIG}{\mmu}\big\}
F\big(x;\mmu,\lambdaIG,-\tfrac{3}{2}\big)\Big|_{x=1}^{\frac{c(t-v)}{u+cv}+1},&
c\geqslant\cS,
\end{cases}
\end{equation*}
and the result follows from Theorems~\ref{sdtyhjtm} and \ref{styhjrtjf}.
\end{proof}

\begin{remark}
We have $\Int{t}{F}(u,c,v)=\underline{O}((u+cv)^{-1})$, as $u+cv\to\infty$.
\end{remark}

In Fig.~\ref{dfgjhkm}, we draw the function $\Int{t}{F}(u,c,v)$ (thick line)
and the corresponding elementary components $(1-Mc)\Elem{t}{1}(u,c,v)$ and
$-\Elem{t}{2}(u,c,v)$ for $v=0$, $M=1$, $D^2=6$, $t=100$, and $u=15$.

%%%%%%%%%%%%%%%%%%%%%%%%%%%%%%%%%%%%%%%%%%%%%%%%%%%%%%%%%%%%%%%%%%%%%%%%%%%%%%%%%%%%%%%%
\subsection{Summand $\Int{t}{S}(u,c,v)=\frac{u+cv}{c^2D^2}\int_{0}^{\frac{c(t-v)}{u+cv}}\frac{(x-\frac{\E{T}}{\E{Y}}c(1+x))^3}{(1+x)^3}
\Ugauss{cM(1+x)}{\frac{c^2D^2(1+x)}{u+cv}}(x)dx$}\label{ertyirtew}
%%%%%%%%%%%%%%%%%%%%%%%%%%%%%%%%%%%%%%%%%%%%%%%%%%%%%%%%%%%%%%%%%%%%%%%%%%%%%%%%%%%%%%%%

\begin{theorem}
For $c>0$, $u>0$, $t>0$, $0<v<t$, we have
\begin{equation*}
\begin{aligned}
\Int{t}{S}(u,c,v)&\eqOK-\frac{3\,c^2D^2}{u+cv}\bigg[\UGauss{0}{1}\Big(\tfrac{\sqrt{u+cv}}{cD\sqrt{x}}
\big(x(1-cM)-1\big)\Big)
\\[-4pt]
&\hskip 28pt+\exp\Big\{2\frac{u+cv}{c^2D^2}(1-cM)\Big\}
\UGauss{0}{1}\Big(-\tfrac{\sqrt{u+cv}}{cD\sqrt{x}}\big(x(1-cM)+1\big)\Big)\bigg]
\bigg|_{x=1}^{\frac{c(t-v)}{u+cv}+1}
\\
&+2(1-cM)\Big(3-4\frac{u+cv}{c^2D^2}(1-cM)\Big)
\\[-4pt]
&\hskip
28pt\times\exp\Big\{2\frac{u+cv}{c^2D^2}(1-cM)\Big\}\UGauss{0}{1}\Big(-\tfrac{\sqrt{u+cv}}{cD\sqrt{x}}
\big(x(1-cM)+1\big)\Big)\bigg|_{x=1}^{\frac{c(t-v)}{u+cv}+1}
\\
&-\frac{\sqrt{2}\,cD}{\sqrt{\pi}\sqrt{u+cv}\,x^{3/2}}\Big(3
\Big(1-\frac{u+cv}{c^2D^2}(1-cM)\Big)x+\frac{u+cv}{c^2D^2}\Big)
\\[-4pt]
&\hskip 28pt\times\exp\Big\{-\frac{u+cv}{2xc^2D^2} \big(x(1-cM)-1\big)^2\Big\}
\bigg|_{x=1}^{\frac{c(t-v)}{u+cv}+1}.
\end{aligned}
\end{equation*}
\end{theorem}

\begin{figure}[t]
\includegraphics[scale=0.8]{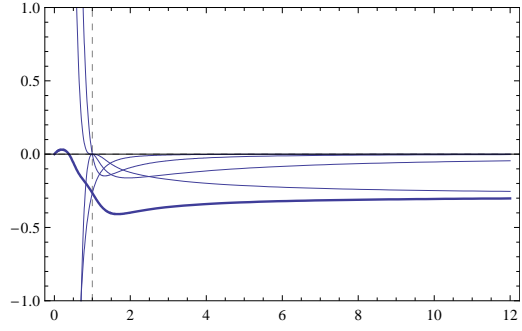}
\caption{\small Graphs ($X$-axis is $c$) of $\Int{t}{S}(u,c,v)$ and of
elementary components in \eqref{dftyjhfgm}. Here $t=100$, $v=0$, $u=15$, $M=1$,
$D^2=6$.}\label{srdtrhyjr}
\end{figure}

\begin{proof}
Observing that $\Int{t}{S}(u,c,v)$ equals the sum of elementary components
\begin{multline}\label{dftyjhfgm}
\frac{(u+cv)}{c^2D^2}(1-Mc)^3\Elem{t}{0}(u,c,v)-3\frac{(u+cv)}{c^2D^2}(1-Mc)^2\Elem{t}{1}(u,c,v)
\\
+3\frac{(u+cv)}{c^2D^2}(1-Mc)\Elem{t}{2}(u,c,v)-\frac{(u+cv)}{c^2D^2}\Elem{t}{3}(u,c,v),
\end{multline}
and bearing in mind Theorems~\ref{sdfghreh}--\ref{rtyhtrjr}, we have
\begin{equation*}
\Int{t}{S}(u,c,v)=\begin{cases}
\frac{\lambdaIG}{\muIG^2}F\big(x;\muIG,\lambdaIG,\tfrac{1}{2}\big)
\Big|_{x=1}^{\frac{c(t-v)}{u+cv}+1}
-\frac{3\lambdaIG}{\muIG^2}F\big(x;\muIG,\lambdaIG,-\tfrac{1}{2}\big)
\Big|_{x=1}^{\frac{c(t-v)}{u+cv}+1}
\\
+\frac{3\lambdaIG(\lambdaIG+\muIG)}{\lambdaIG\muIG^2}
F\big(x;\muIG,\lambdaIG,-\tfrac{3}{2}\big)\Big|_{x=1}^{\frac{c(t-v)}{u+cv}+1}
\\
-\frac{\lambdaIG^2+3\lambdaIG
\muIG+3\muIG^2}{\lambdaIG\muIG^2}F\big(x;\muIG,\lambdaIG,-\tfrac{5}{2}\big)
\Big|_{x=1}^{\frac{c(t-v)}{u+cv}+1} ,&0<c\leqslant\cS,
\\[8pt]
-\frac{\lambdaIG}{\mmu^2}\exp\big\{-2\,\tfrac{\lambdaIG}{\mmu}\big\}
F\big(x;\mmu,\lambdaIG,\tfrac{1}{2}\big)\Big|_{x=1}^{\frac{c(t-v)}{u+cv}+1}
\\
-\frac{3\lambdaIG}{\mmu^2}\exp\big\{-2\,\tfrac{\lambdaIG}{\mmu}\big\}
F\big(x;\mmu,\lambdaIG,-\tfrac{1}{2}\big)\Big|_{x=1}^{\frac{c(t-v)}{u+cv}+1}
\\
-\frac{3\lambdaIG(\lambdaIG+\mmu)}{\lambdaIG\mmu^2}\exp\big\{-2\frac{\lambdaIG}{\mmu}\big\}
F\big(x;\mmu,\lambdaIG,-\tfrac{3}{2}\big)\Big|_{x=1}^{\frac{c(t-v)}{u+cv}+1}
\\
-\frac{\lambdaIG^2+3\lambdaIG\mmu+3\mmu^2}{\lambdaIG\mmu^2}\exp\big\{-2\frac{\lambdaIG}{\mmu}\big\}
F\big(x;\mmu,\lambdaIG,-\tfrac{3}{2}\big)\Big|_{x=1}^{\frac{c(t-v)}{u+cv}+1},&c\geqslant\cS,
\end{cases}
\end{equation*}
and the result follows from Theorems~\ref{sdrwhgtrjn}--\ref{wdrthytm} by means
of tedious but straightforward calculations.
\end{proof}

\begin{remark}
We have $\Int{t}{S}(u,c,v)=\underline{O}((u+cv)^{-1})$, as $u+cv\to\infty$.
\end{remark}

In Fig.~\ref{srdtrhyjr}, we draw the function $\Int{t}{F}(u,c,v)$ (thick line)
and the corresponding elementary components (see \eqref{dftyjhfgm}) for $v=0$,
$M=1$, $D^2=6$, $t=100$, and $u=15$.

%%%%%%%%%%%%%%%%%%%%%%%%%%%%%%%%%%%%%%%%%%%%%%%%%%%%%%%%%%%%%%%%%%%%%%%%%%%%%%%%%%%%%%%%
\section{Proof of Theorems~\ref{srdthjrf} and~\ref{dthyjrt}}\label{weryh5ej}
%%%%%%%%%%%%%%%%%%%%%%%%%%%%%%%%%%%%%%%%%%%%%%%%%%%%%%%%%%%%%%%%%%%%%%%%%%%%%%%%%%%%%%%%

We start with the proof of Theorem~\ref{srdthjrf}. This proof relies on, and is
built over the proof in \citeNP{[Malinovskii 2017]}. Because of the limited
volume, we greatly reduce exposition of those its parts which may be found in
detail\footnote{Mainly, it relates to evaluation of the residual terms.} in
\citeNP{[Malinovskii 2017]}. It mainly refers to estimation of residual terms.
We will focus on those parts that are new and which allow us to construct more
accurate approximation. As before, key formula\footnote{It is equation (1.4) in
\citeNP{[Malinovskii 2017]}.} is
\begin{equation}\label{qwretgjhmg}
\P\{v<\timeR\leqslant
t\mid\T{1}=v\}=\int_{v}^{t}\dfrac{u+cv}{u+cz}\sum_{n=1}^{\infty}
\P\big\{M(u+cz)=n\big\}f_{T}^{*n}(z-v)dz,
\end{equation}
where $M(s)=\inf\{k\geqslant 1:\sum_{i=1}^{k}Y_{i}>s\}-1$. We put $y=z-v$ in
\eqref{qwretgjhmg} and rewrite it as\footnote{It is equation (6.1) in
\citeNP{[Malinovskii 2017]}.}
\begin{equation}\label{wertyrjhn}
\begin{aligned}
\P\{v<\timeR\leqslant t\mid\T{1}=v\}&=\int_{0}^{t-v}\dfrac{u+cv}{u+cv+cy}\,\,
\p_{\big\{\sum_{i=2}^{M(u+cv+cy)+1}\T{i}\big\}}(y)dy
\\[4pt]
&=\sum_{n=1}^{\infty}\int_{0}^{t-v}\frac{u+cv}{u+cv+cy}
\int_{0}^{u+cv+cy}\P\left\{\Y{n+1}>z\right\}
\\[2pt]
&\hskip 90pt\times f_{Y}^{*n}(u+cv+cy-z)f_{T}^{*n}(y)dydz.
\end{aligned}
\end{equation}

Bearing in mind that $\T{i}$, $i=1,2,\dots$ and $\Y{i}$, $i=1,2,\dots$ are
mutually independent, the second equality in \eqref{qwretgjhmg} holds true
since
\begin{multline}\label{dtuhjtfuky}
\P\{M(u+cv+cy)=n\}=\P\bigg\{\sum_{i=1}^{n}\Y{i}\leqslant
u+cv+cy<\sum_{i=1}^{n+1}\Y{i}\bigg\}
\\[4pt]
=\int_{0}^{u+cv+cy}f^{*n}_{Y}(u+cv+cy-z)\P\{\Y{n+1}>z\}dz.
\end{multline}

The proof consists of several steps. The steps similar to Steps 1 and 3 in
\citeNP{[Malinovskii 2017]} are technical and aim elimination of the terms that
have little impact in \eqref{wertyrjhn}; it may be called preparation of
\eqref{wertyrjhn} for further analysis. It is much the same thing as in
\citeNP{[Malinovskii 2017]}, and we will not repeat the details. We merely
recall that Step 1 aims rejection of terms that correspond to small $n$, for
which the event $\{M(u+cv+cy)=n\}$ has a small probability, as $u+cv+cy$ is
large. On this step, we use bounds for probabilities of large deviations of
sums of i.i.d. random variables, like in \citeNP{[Nagaev 1965]}. Step 3 aims
processing of terms that contain $z$, i.e., defect of the random walk
$\sum_{i=1}^{n}Y_i$, $n=1,2,\dots$, as it crosses the level $u+cv+cy$ (see
\eqref{dtuhjtfuky}). This is based on application of Taylor formula, and we
discuss it below in more detail.

The step similar to Step 2 in \citeNP{[Malinovskii 2017]} consists in
application to the product $f_{Y}^{*n}(u+cv+cy-z)f_{T}^{*n}(y)$ in
\eqref{wertyrjhn} of Edgeworth expansions\footnote{In contrast to non-uniform
Berry-Esseen bounds in local CLT, as in \citeNP{[Malinovskii 2017]}.} in the
local central limit theorem (CLT) with non-uniform remainder term. It yields
main, correction, and residual terms of the approximation in a raw form. The
rest of the proof, which also consists of several steps, is
elaboration\footnote{This means simplification, or transformation, when
discarded are the terms of allowed order smallness. It will be seen below that
the main tool on this way will be a representation of the sums in the form of
integral sums and their approximation by the corresponding integrals.} of all
these terms, provided that the required accuracy is always held.

%%%%%%%%%%%%%%%%%%%%%%%%%%%%%%%%%%%%%%%%%%%%%%%%%%%%%%%%%%%%%%%%%%%%%%%%%%%%%%%%%%%%%%%%
\subsection{Use of Edgeworth expansions in CLT}\label{sdfyguklyi}
%%%%%%%%%%%%%%%%%%%%%%%%%%%%%%%%%%%%%%%%%%%%%%%%%%%%%%%%%%%%%%%%%%%%%%%%%%%%%%%%%%%%%%%%

For $\Y{i}\overset{d}{=}Y$ and $\T{i}\overset{d}{=}T$, let us introduce the
standardized random variables
$\tilde{Y}_i=(Y_i-\E{Y})/\sqrt{\D{Y}}\overset{d}{=}\tilde{Y}$ and
$\tilde{T}_i=(T_i-\E{T})/\sqrt{\D{T}}\overset{d}{=}\tilde{T}$. It is noteworthy
that, e.g., $\E\tilde{Y}^3=\E(Y-\E{Y})^3/(\D{Y})^{3/2}$ and
$\E\tilde{T}^3=\E(T-\E{T})^3/(\D{T})^{3/2}$. For i.i.d. random vectors
$\xi_i=(\tilde{Y}_i,\tilde{T}_i)\in\R^{2}$, we bear in mind that
\begin{equation*}
f_{T}^{*n}(x)=\frac{1}{\sqrt{n\D{T}}}\,p_{n^{-1/2}\sum_{i=1}^{n}\tilde{T}_i}
\left(\tfrac{x-n\E{T}}{\sqrt{n\D{T}}}\right),\quad
f_{Y}^{*n}(x)=\frac{1}{\sqrt{n\D{Y}}}\,p_{n^{-1/2}\sum_{i=1}^{n}\tilde{Y}_i}
\left(\tfrac{x-n\E{Y}}{\sqrt{n\D{Y}}}\right),
\end{equation*}
and take advantage of Theorem~\ref{w4e5y46yu43}. In this way, we have
\begin{equation}\label{werkyhjrtjr}
\big|\,\P\{v<\timeR\leqslant
t\mid\T{1}=v\}-\expaN{t}(u,c,v)\big|\leqslant\MainRemTerm{t}(u,c,v),
\end{equation}
where\footnote{Quite the same as in Step 1 in \citeNP{[Malinovskii 2017]}, we
have reduced first the area of summation, rejecting terms that correspond to
$n\leqslant\EnOne=\epsilon(u+cv)$, where $0<\epsilon<1$. In the use of
estimates like in \citeNP{[Malinovskii 2017]}, we bear in mind that
$\E{T}^4<\infty$, $\E{Y}^4<\infty$.} $c,u>0$, $0<v<t$,
\begin{equation}\label{wert45hjrt6jn}
\expaN{t}(u,c,v)=\expaN{t}^{[1]}(u,c,v)+\expaN{t}^{[2]}(u,c,v)
+\expaN{t}^{[3]}(u,c,v)
\end{equation}
with
\begin{equation*}
\begin{aligned}
\expaN{t}^{[1]}(u,c,v)&=\frac{1}{\sqrt{\D{Y}\D{T}}}\,
\sum_{n=\EnOne}^{\infty}n^{-1}\int_{0}^{t-v}\frac{u+cv}{u+cv+cy}
\int_{0}^{u+cv+cy}\P\left\{\Y{n+1}>z\right\}
\\
&\hskip
0pt\times\Ugauss{0}{1}\bigg(\frac{u+cv+cy-z-n\E{Y}}{\sqrt{n\D{Y}}}\bigg)
\Ugauss{0}{1}\bigg(\frac{y-n\E{T}}{\sqrt{n\D{T}}}\bigg)\,dydz,
\\[4pt]
\expaN{t}^{[2]}(u,c,v)&=\frac{\E\tilde{T}^3}{6\sqrt{\D{Y}\D{T}}}
\sum_{n=\EnOne}^{\infty}n^{-3/2}\int_{0}^{t-v}\frac{u+cv}{u+cv+cy}
\int_{0}^{u+cv+cy}\P\left\{\Y{n+1}>z\right\}
\\
&\hskip 0pt\times\bigg(\bigg(\dfrac{y-n\E{T}}{\sqrt{n\D{T}}}\bigg)^3
-3\bigg(\dfrac{y-n\E{T}}{\sqrt{n\D{T}}}\bigg)\bigg)
\\
&\times\Ugauss{0}{1}\bigg(\frac{u+cv+cy-z-n\E{Y}}{\sqrt{n\D{Y}}}\bigg)\Ugauss{0}{1}\bigg(\dfrac{y-n\E{T}}
{\sqrt{n\D{T}}}\bigg)\,dydz,
\\[4pt]
\expaN{t}^{[3]}(u,c,v)&=\frac{\E{\tilde{Y}^3}}{6\sqrt{\D{Y}\D{T}}}
\sum_{n=\EnOne}^{\infty}n^{-3/2}\int_{0}^{t-v}\frac{u+cv}{u+cv+cy}
\int_{0}^{u+cv+cy}\P\left\{\Y{n+1}>z\right\}
\\
&\hskip 0pt\times\bigg(\bigg(\dfrac{u+cv+cy-z-n\E{Y}}{\sqrt{n\D{Y}}}\bigg)^3-3
\bigg(\dfrac{u+cv+cy-z-n\E{Y}}{\sqrt{n\D{Y}}}\bigg)\bigg)
\\
&\times\Ugauss{0}{1}\bigg(\dfrac{u+cv+cy-z-n\E{Y}}{\sqrt{n\D{Y}}}\bigg)\Ugauss{0}{1}\bigg(\dfrac{y-n\E{T}}
{\sqrt{n\D{T}}}\bigg)\,dydz.
\end{aligned}
\end{equation*}
%\begin{multline*}
%\expaN{t}(u,c,v)
%=\sum_{n=\EnOne}^{\infty}\int_{0}^{t-v}\frac{u+cv}{u+cv+cy}
%\int_{0}^{u+cv+cy}\P\left\{\Y{n+1}>z\right\}
%\\
%\times\frac{1}{\sqrt{n\D{Y}}}\frac{1}{\sqrt{n\D{T}}} \Ugauss{0}{1}
%\bigg(\dfrac{u+cv+cy-z-n\E{Y}}{\sqrt{n\D{Y}}}\bigg)\Ugauss{0}{1}\bigg(\dfrac{y-n\E{T}}
%{\sqrt{n\D{T}}}\bigg)
%\\
%\times\bigg\{1+n^{-1/2}\bigg[\frac{\E{\tilde{Y}^3}}{6}
%\bigg(\bigg(\dfrac{u+cv+cy-z-n\E{Y}}{\sqrt{n\D{Y}}}\bigg)^3
%-3\bigg(\dfrac{u+cv+cy-z-n\E{Y}}{\sqrt{n\D{Y}}}\bigg)\bigg)
%\\
%+\bigg(\frac{\E\tilde{T}^3}{6}\bigg(\bigg(\dfrac{y-n\E{T}}{\sqrt{n\D{T}}}\bigg)^3
%-3\bigg(\dfrac{y-n\E{T}}{\sqrt{n\D{T}}}\bigg)\bigg)\bigg)\bigg]\bigg\}\,dydz,
%\end{multline*}
and
\begin{multline*}
\MainRemTerm{t}(u,c,v)=K\sum_{n=\EnOne}^{\infty}n^{-2}\int_{0}^{t-v}\frac{u+cv}{u+cv+cy}
\int_{0}^{u+cv+cy}\P\left\{\Y{n+1}>z\right\}
\\
\times\bigg(1+\bigg[\bigg(\frac{u+cv+cy-z-n\E{Y}}{\sqrt{n\D{Y}}}\bigg)^2
+\bigg(\frac{y-n\E{T}}{\sqrt{n\D{T}}}\bigg)^2\bigg]^{1/2}\bigg)^{-4}\,dydz.
\end{multline*}

%The approximating term rewrites as
%\begin{equation}\label{wert45hjrt6jn}
%\expaN{t}(u,c,v)=\expaN{t}^{[1]}(u,c,v)+\expaN{t}^{[2]}(u,c,v)
%+\expaN{t}^{[3]}(u,c,v),
%\end{equation}
%where

\begin{remark}[Use of two-di\-men\-si\-onal local CLT]\label{wqeyhjr}
Deriving \eqref{werkyhjrtjr}, we applied Theorem~\ref{w4e5y46yu43} to the
product $f_{Y}^{*n}(u+cv+cy-z)f_{T}^{*n}(y)$. It is Edgeworth expansions in
two-di\-men\-si\-onal local CLT. Alternatively, we could consider
$f_{Y}^{*n}(u+cv+cy-z)$ and $f_{T}^{*n}(y)$ one-by-one, separately, by applying
Edgeworth expansions in one-di\-men\-si\-onal local CLT to each of these
factors. We preferred to use Theorem~\ref{w4e5y46yu43} to get the remainder
term $\MainRemTerm{t}(u,c,v)$ in a form better suited for further analysis.
\end{remark}

%%%%%%%%%%%%%%%%%%%%%%%%%%%%%%%%%%%%%%%%%%%%%%%%%%%%%%%%%%%%%%%%%%%%%%%%%%%%%%%%%%%%%%%%
\subsection{Reducing of approximation \eqref{wert45hjrt6jn} to a convenient form}\label{wertwserewye}
%%%%%%%%%%%%%%%%%%%%%%%%%%%%%%%%%%%%%%%%%%%%%%%%%%%%%%%%%%%%%%%%%%%%%%%%%%%%%%%%%%%%%%%%

On this step, we proceed in the same way as in \citeNP{[Malinovskii 2017]}.
First, we make a suitable change of variables. We put $x=c\,y/(u+cv)$,
$dx=c\,dy/(u+cv)$, and
\begin{equation*}
\mathcal{Y}_{n,z}(u+cv,x)=\dfrac{(u+cv)(1+x)-z-n\E{Y}}{\sqrt{n\D{Y}}},\hskip
6pt\mathcal{T}_n(u+cv,x)=\dfrac{(u+cv)x/c-n\E{T}}{\sqrt{n\D{T}}}.
\end{equation*}
We represent the summands in \eqref{wert45hjrt6jn} as follows\footnote{We bear
in mind that $Y_{n+1}\overset{d}{=}Y$ and that $cy=(u+cv)x$, $cdy=(u+cv)dx$.}:
\begin{equation*}
\begin{aligned}
\expaN{t}^{[1]}(u,c,v)&=\frac{u+cv}{c\sqrt{\D{Y}\D{T}}}
\,\int_{0}^{\frac{c(t-v)}{u+cv}}
\frac{1}{1+x}\int_{0}^{(u+cv)(1+x)}\P\left\{Y>z\right\}
\\
&\hskip
0pt\times\,\sum_{n=\EnOne}^{\infty}n^{-1}\Ugauss{0}{1}\big(\mathcal{Y}_{n,z}(u+cv,x)\big)
\Ugauss{0}{1}\big(\mathcal{T}_n(u+cv,x)\big)\,dxdz,
\\
\expaN{t}^{[2]}(u,c,v)&=\frac{(u+cv)\E(\tilde{T}^3)}{6c\sqrt{\D{Y}\D{T}}}
\int_{0}^{\frac{c(t-v)}{u+cv}}
\frac{1}{1+x}\int_{0}^{(u+cv)(1+x)}\P\left\{Y>z\right\}
\\[4pt]
&\hskip
0pt\times\sum_{n=\EnOne}^{\infty}n^{-3/2}\big(\mathcal{T}^3_n(u+cv,x)-3\mathcal{T}_n(u+cv,x)\big)
\\[4pt]
&\times\Ugauss{0}{1}\big(\mathcal{Y}_{n,z}(u+cv,x)\big)\Ugauss{0}{1}
\big(\mathcal{T}_n(u+cv,x)\big)\,dxdz,
%\\
\end{aligned}
\end{equation*}
\begin{equation*}
\begin{aligned}
\expaN{t}^{[3]}(u,c,v)&=\frac{(u+cv)\E(\tilde{Y}^3)}{6c\sqrt{\D{Y}\D{T}}}
\int_{0}^{\frac{c(t-v)}{u+cv}}
\frac{1}{1+x}\int_{0}^{(u+cv)(1+x)}\P\left\{Y>z\right\}
\\[4pt]
&\hskip 0pt\times
\sum_{n=\EnOne}^{\infty}n^{-3/2}\big(\mathcal{Y}^3_{n,z}(u+cv,x)-3\mathcal{Y}_{n,z}(u+cv,x)\big)
\\[4pt]
&\times\Ugauss{0}{1}\big(\mathcal{Y}_{n,z}(u+cv,x)\big)
\Ugauss{0}{1}\big(\mathcal{T}_n(u+cv,x)\big)\,dxdz.
\end{aligned}
\end{equation*}
In the same way, we write
\begin{equation*}
\begin{aligned}
\MainRemTerm{t}(u,c,v)&=K(u+cv)\int_{0}^{\frac{c(t-v)}{u+cv}}
\frac{1}{1+x}\int_{0}^{(u+cv)(1+x)}\P\left\{Y>z\right\}
\\
&\times\sum_{n=\EnOne}^{\infty}n^{-2}\big(1+\big[\big(\mathcal{Y}_{n,z}(u+cv,x)\big)^2
+\big(\mathcal{T}_n(u+cv,x)\big)^2\big]^{1/2}\big)^{-4}\,dxdz.
\end{aligned}
\end{equation*}

Second, we develop and use the extensions of fundamental identities of
Section~\ref{werwhteyjn}. We set
\begin{equation}\label{wergeghwerh}
\begin{aligned}
\Delta_{n,z}(u+cv,x)&=\dfrac{(u+cv)(x/c)\E{Y}-[(u+cv)(1+x)-z]\E{T}}{\sqrt{B_1n}},
\\
\Lambda_{n,z}(u+cv,x)
&=\dfrac{B_1n-\big(B_2[(u+cv)(1+x)-z]+B_3(u+cv)x/c\big)}{\sqrt{B_1B_4n}},
\end{aligned}
\end{equation}
where $B_1=(\E{T})^2\D{Y}+(\E{Y})^2\D{T}$, $B_2=\E{Y}\D{T}$, $B_3=\E{T}\D{Y}$,
and $B_4=\D{Y}\D{T}$, and apply the following identities which proof is
straightforward.

\begin{lemma}\label{ewrgfwdgsgw}
The following identities hold true:
\begin{equation*}
\begin{aligned}
\mathcal{Y}_{n,z}(u+cv,x)
&=-\frac{\E{Y}}{\sqrt{\D{Y}B_1}}\Big(\sqrt{B_4}\Lambda_{n,z}(u+cv,x)+\frac{B_3}{\E{Y}}\Delta_{n,z}(u+cv,x)\Big),
\\
\mathcal{T}_{n,z}(u+cv,x)&=\frac{\E{T}}{\sqrt{\D{T}B_1}}
\Big(\frac{B_2}{\E{T}}\Delta_{n,z}(u+cv,x)-\sqrt{B_4}\Lambda_{n,z}(u+cv,x)\Big).
\end{aligned}
\end{equation*}
\end{lemma}

The following Lemmas \ref{rtyujrtkj}, \ref{ewrtyjhrmt} are straightforward from
Lemma \ref{ewrgfwdgsgw}.

\begin{lemma}\label{rtyujrtkj}
The following identities hold true:
\begin{equation*}
\mathcal{Y}^2_{n,z}(u+cv,x)+\mathcal{T}^2_n(u+cv,x)=\Lambda_{n,z}^2(u+cv,x)+\Delta^2_{n,z}(u+cv,x),
\end{equation*}
and
\begin{multline*}
\Ugauss{0}{1}\big(\mathcal{Y}_{n,z}(u+cv,x)\big)\Ugauss{0}{1}
\big(\mathcal{T}_n(u+cv,x)\big)
\\
=\frac{1}{2\pi}
\exp\big\{-\tfrac{1}{2}\big[\Lambda_{n,z}^2(u+cv,x)+\Delta^2_{n,z}(u+cv,x)\big]\big\}.
\end{multline*}
\end{lemma}

\begin{lemma}\label{ewrtyjhrmt}
The following identities hold true:
\begin{multline*}
\mathcal{Y}^3_{n,z}(u+cv,x)-3\mathcal{Y}_{n,z}(u+cv,x)
=-\bigg(\frac{\E{Y}}{\sqrt{\D{Y}B_1}}\bigg)^3\bigg(\frac{B^3_3}{(\E{Y})^3}\Delta_{n,z}^3(u+cv,x)
\\
+3\frac{B^2_3\sqrt{B_4}}{(\E{Y})^2}\Delta_{n,z}^2(u+cv,x)\Lambda_{n,z}(u+cv,x)
+3\frac{B_3B_4}{\E{Y}}\Delta_{n,z}(u+cv,x)\Lambda^2_{n,z}(u+cv,x)
\\
+B_4^{3/2}\Lambda^3_{n,z}(u+cv,x)\bigg)
+3\frac{\E{Y}}{\sqrt{\D{Y}B_1}}\Big(\sqrt{B_4}\Lambda_{n,z}(u+cv,x)
+\frac{B_3}{\E{Y}}\Delta_{n,z}(u+cv,x)\Big),
\end{multline*}
and
\begin{multline*}
\mathcal{T}^3_{n,z}(u+cv,x)-3\mathcal{T}_{n,z}(u+cv,x)
=\bigg(\frac{\E{T}}{\sqrt{\D{T}B_1}}\bigg)^3
\Big(\frac{B_2^3}{(\E{T})^3}\Delta^3_{n,z}(u+cv,x)
\\
-3\frac{B_2^2\sqrt{B_4}}{(\E{T})^2}\Delta^2_{n,z}(u+cv,x)\Lambda_{n,z}(u+cv,x)
+3\frac{B_2B_4}{\E{T}}\Delta_{n,z}(u+cv,x)\Lambda^2_{n}(u+cv,x)
\\
-B_4^{3/2}\Lambda^3_{n,z}(u+cv,x)\Big)-3\frac{\E{T}}{\sqrt{\D{T}B_1}}
\Big(\frac{B_2}{\E{T}}\Delta_{n,z}(u+cv,x)-\sqrt{B_4}\Lambda_{n,z}(u+cv,x)\Big).
\end{multline*}
\end{lemma}

%%%%%%%%%%%%%%%%%%%%%%%%%%%%%%%%%%%%%%%%%%%%%%%%%%%%%%%%%%%%%%%%%%%%%%%%%%%%%%%%%%%%%%%%
%\subsubsection{Use of fundamental identities of Section~\ref{werwhteyjn}}\label{wertrjtyh}
%%%%%%%%%%%%%%%%%%%%%%%%%%%%%%%%%%%%%%%%%%%%%%%%%%%%%%%%%%%%%%%%%%%%%%%%%%%%%%%%%%%%%%%%

Thus, using identities of Lemmas~\ref{ewrgfwdgsgw}, \ref{rtyujrtkj} and of
Lemma~\ref{ewrtyjhrmt}, along with fundamental identities of
Section~\ref{werwhteyjn}, we write
\begin{multline*}
\expaN{t}^{[1]}(u,c,v)=\frac{u+cv}{2\pi
c\sqrt{\D{T}\D{Y}}}\int_{0}^{\frac{c(t-v)}{u+cv}}
\frac{1}{1+x}\int_{0}^{(u+cv)(1+x)}\P\left\{Y>z\right\}\sum_{n=\EnOne}^{\infty}n^{-1}
\\
\times\exp\big\{-\tfrac{1}{2}\big[\Lambda_{n,z}^2(u+cv,x)+\Delta^2_{n,z}(u+cv,x)\big]\big\}dxdz,
\end{multline*}
\begin{multline*}
\expaN{t}^{[2]}(u,c,v)=\frac{(u+cv)\E(\tilde{T}^3)}{12\pi c\sqrt{\D{Y}\D{T}}}
\int_{0}^{\frac{c(t-v)}{u+cv}}
\frac{1}{1+x}\int_{0}^{(u+cv)(1+x)}\P\left\{Y>z\right\}\sum_{n=\EnOne}^{\infty}n^{-3/2}
\\[0pt]
\times\bigg\{\bigg(\frac{\E{T}}{\sqrt{\D{T}B_1}}\bigg)^3
\bigg(\frac{B_2^3}{(\E{T})^3}\Delta^3_{n,z}(u+cv,x)
-3\frac{B_2^2\sqrt{B_4}}{(\E{T})^2}\Delta^2_{n,z}(u+cv,x)\Lambda_{n,z}(u+cv,x)
\\
+3\frac{B_2B_4}{\E{T}}\Delta_{n,z}(u+cv,x)\Lambda^2_{n}(u+cv,x)
-B_4^{3/2}\Lambda^3_{n,z}(u+cv,x)\bigg)
\\
-3\frac{\E{T}}{\sqrt{\D{T}B_1}}
\bigg(\frac{B_2}{\E{T}}\Delta_{n,z}(u+cv,x)-\sqrt{B_4}\Lambda_{n,z}(u+cv,x)\bigg)\bigg\}
\\[4pt]
\times\exp\big\{-\tfrac{1}{2}\big[\Lambda_{n,z}^2(u+cv,x)+\Delta^2_{n,z}(u+cv,x)\big]\big\}dxdz,
\end{multline*}
\begin{multline*}
\expaN{t}^{[3]}(u,c,v)=-\frac{(u+cv)\E({\tilde{Y}^3})}{12\pi
c\sqrt{\D{Y}\D{T}}} \int_{0}^{\frac{c(t-v)}{u+cv}}
\frac{1}{1+x}\int_{0}^{(u+cv)(1+x)}\P\left\{Y>z\right\}\sum_{n=\EnOne}^{\infty}n^{-3/2}
\\[4pt]
\times\bigg\{\bigg(\frac{\E{Y}}{\sqrt{\D{Y}B_1}}\bigg)^3\Big(\frac{B^3_3}{(\E{Y})^3}\Delta_{n,z}^3(u+cv,x)
+3\frac{B^2_3\sqrt{B_4}}{(\E{Y})^2}\Delta_{n,z}^2(u+cv,x)\Lambda_{n,z}(u+cv,x)
\\
+3\frac{B_3B_4}{\E{Y}}\Delta_{n,z}(u+cv,x)\Lambda^2_{n,z}(u+cv,x)
+B_4^{3/2}\Lambda^3_{n,z}(u+cv,x)\Big)
\\
-3\frac{\E{Y}}{\sqrt{\D{Y}B_1}}\Big(\sqrt{B_4}\Lambda_{n,z}(u+cv,x)
+\frac{B_3}{\E{Y}}\Delta_{n,z}(u+cv,x)\Big)\bigg\}
\\[4pt]
\times\exp\big\{-\tfrac{1}{2}\big[\Lambda_{n,z}^2(u+cv,x)+\Delta^2_{n,z}(u+cv,x)\big]\big\}dxdz,
\end{multline*}
and
\begin{multline*}
\MainRemTerm{t}(u,c,v)=K(u+cv)\int_{0}^{\frac{c(t-v)}{u+cv}}\frac{1}{1+x}
\int_{0}^{(u+cv)(1+x)}\P\left\{Y>z\right\}\,
\\
\times\sum_{n=\EnOne}^{\infty}n^{-2}\big(1+\big[\Lambda_{n,z}^2(u+cv,x)+\Delta^2_{n,z}(u+cv,x)\big]^{1/2}\big)^{-4}dxdz.
\end{multline*}

%%%%%%%%%%%%%%%%%%%%%%%%%%%%%%%%%%%%%%%%%%%%%%%%%%%%%%%%%%%%%%%%%%%%%%%%%%%%%%%%%%%%%%%%
\subsection{Processing of terms that contain $z$ in $\expaN{t}^{[1]}(u,c,v)$--$\expaN{t}^{[1]}(u,c,v)$ and in $\MainRemTerm{t}(u,c,v)$}\label{rytjukty}
%%%%%%%%%%%%%%%%%%%%%%%%%%%%%%%%%%%%%%%%%%%%%%%%%%%%%%%%%%%%%%%%%%%%%%%%%%%%%%%%%%%%%%%%

The same way as in \citeNP{[Malinovskii 2017]}, we will discard the terms
containing $z$, i.e., defect of the random walk $\sum_{i=1}^{n}Y_i$,
$n=1,2,\dots$, holding the allowed accuracy of approximation. We rewrite
\eqref{wergeghwerh} as
\begin{equation*}
\Delta_{n,z}(u+cv,x)=\Delta_{n}(u+cv,x) +\frac{z\E{T}}{\sqrt{B_1n}}, \quad
\Lambda_{n,z}(u+cv,x)=\Lambda_{n}(u+cv,x) +\frac{zB_2}{\sqrt{B_1B_4n}},
\end{equation*}
where
\begin{equation*}
\begin{aligned}
\Delta_{n}(u+cv,x)&=(u+cv)\dfrac{(x/c)\E{Y}-(1+x)\E{T}}{\sqrt{B_1n}},
\\
\Lambda_{n}(u+cv,x)
&=\dfrac{B_1n-(B_2(u+cv)(1+x)+B_3(u+cv)(x/c))}{\sqrt{B_1B_4n}}.
\end{aligned}
\end{equation*}

\begin{lemma}[Taylor's formula for exponential term]\label{dfgfmgmh}
We have\footnote{Here and in some subsequent stages of the proof certain
cumbersome but evident formulas are skipped and replaced by ellipsis. This
refers mainly to remainder terms as here, and is done because of the volume
restrictions. The reader will easily restore the skipped formulas.}
\begin{multline*}
\exp\big\{-\tfrac{1}{2}\big[\Lambda_{n,z}^2(u+cv,x)+\Delta^2_{n,z}(u+cv,x)\big]\big\}
=\exp\big\{-\tfrac{1}{2}\big[\Lambda_{n}^2(u+cv,x)+\Delta^2_{n}(u+cv,x)\big]\big\}
\\[4pt]
\times\big(1+z(\D{Y}B_1n)^{-1/2}\big(\E{T}{\sqrt{\D{Y}}}\Delta_{n}(u+cv,x)
+\E{Y}\sqrt{\D{T}}\Lambda_{n}(u+cv,x)\big)+\dots\big).
\end{multline*}
\end{lemma}

\begin{proof}
For the proof, we apply Taylor's theorem $f(z)=f(0)+zf^{\prime}(0)+\dots$ to
the function
$f(z)=\exp\big\{-\tfrac{1}{2}\big[\Delta^2_{n,z}(u+cv,x)+\Lambda_{n,z}^2(u+cv,x)\big]\big\}$.
Bearing in mind that
\begin{equation*}
f^{\prime}(z)=f(z)(\D{Y}B_1n)^{-1/2}\big(\E{T}{\sqrt{\D{Y}}}\Delta_{n,z}
+\E{Y}\sqrt{\D{T}}\Lambda_{n,z}\big),
\end{equation*}
we have the result.
\end{proof}

\begin{lemma}\label{sdrtfhfgjmf}
We have
\begin{equation*}
\begin{gathered}
0\leqslant\E{Y}-\int_{0}^{(u+cv)(1+x)}\P\left\{Y>z\right\}dz
\leqslant\frac{\E(Y^4)}{3((u+cv)(1+x))^3},
\\
0\leqslant\frac{\E({Y^2})}{2}-\int_{0}^{(u+cv)(1+x)}z\P\left\{Y>z\right\}dz
\leqslant\frac{\E(Y^4)}{2((u+cv)(1+x))^2}.
\end{gathered}
\end{equation*}
\end{lemma}

\begin{proof}
The proof of Lemma~\ref{sdrtfhfgjmf} is straightforward from easy equalities
\begin{equation*}
\begin{gathered}
\E{Y}-\int_{0}^{(u+cv)(1+x)}\P\left\{Y>z\right\}dz=\int_{(u+cv)(1+x)}^{\infty}
\P\left\{Y>z\right\}dz,
\\
\frac{\E({Y^2})}{2}-\int_{0}^{(u+cv)(1+x)}z\P\left\{Y>z\right\}dz
=\int_{(u+cv)(1+x)}^{\infty}z\P\left\{Y>z\right\}dz,
\end{gathered}
\end{equation*}
and from Chebychev's inequality
$\P\left\{Y>z\right\}\leqslant\dfrac{\E(Y^4)}{z^{4}}$.
\end{proof}

%%%%%%%%%%%%%%%%%%%%%%%%%%%%%%%%%%%%%%%%%%%%%%%%%%%%%%%%%%%%%%%%%%%%%%%%%%%%%%%%%%%%%%%%
%\subsection*{Reduction of $\expaN{t}^{[1]}(u,c,v)$--$\expaN{t}^{[3]}(u,c,v)$}\label{rtyhjrtjyt}
%%%%%%%%%%%%%%%%%%%%%%%%%%%%%%%%%%%%%%%%%%%%%%%%%%%%%%%%%%%%%%%%%%%%%%%%%%%%%%%%%%%%%%%%

Applying Lemmas~\ref{dfgfmgmh} and~\ref{sdrtfhfgjmf} to
$\expaN{t}^{[1]}(u,c,v)$, we reduce it to the sum
$\MainApprox{t}^{[1]}(u,c,v)+\CoR{t}^{[1]}(u,c,v)$, where
\begin{equation}\label{wsrdtgejhr}
\begin{aligned}
\MainApprox{t}^{[1]}(u,c,v)=\frac{(u+cv)\E{Y}}{2\pi
c\sqrt{\D{T}\D{Y}}}&\int_{0}^{\frac{c(t-v)}{u+cv}}
\frac{1}{1+x}\sum_{n=\EnOne}^{\infty}n^{-1}
\\[0pt]
&\times\exp\Big\{-\tfrac{1}{2}\Big[\Lambda_{n}^2(u+cv,x)+\Delta^2_{n}(u+cv,x)\Big]\Big\}dx
\end{aligned}
\end{equation}
and
\begin{equation}\label{2345ty45uj4r}
\CoR{t}^{[1]}(u,c,v)=\CoR{t}^{[1,1]}(u,c,v)+\CoR{t}^{[1,2]}(u,c,v),
\end{equation}
where
\begin{equation*}
\begin{aligned}
\CoR{t}^{[1,1]}(u,c,v)=K^{[1,1]}&\int_{0}^{\frac{c(t-v)}{u+cv}}
\frac{1}{1+x}\sum_{n=\EnOne}^{\infty}n^{-3/2}\Delta_{n}(u+cv,x)
\\[0pt]
&\times\exp\Big\{-\tfrac{1}{2}\Big[\Lambda_{n}^2(u+cv,x)+\Delta^2_{n}(u+cv,x)\Big]\Big\}dx,
\\[0pt]
\CoR{t}^{[1,2]}(u,c,v)=K^{[1,2]}&\int_{0}^{\frac{c(t-v)}{u+cv}}
\frac{1}{1+x}\sum_{n=\EnOne}^{\infty}n^{-3/2}\Lambda_{n}(u+cv,x)
\\[0pt]
&\times\exp\Big\{-\tfrac{1}{2}\Big[\Lambda_{n}^2(u+cv,x)+\Delta^2_{n}(u+cv,x)\Big]\Big\}dx
\end{aligned}
\end{equation*}
with%\marginpar{CHECK!}
\begin{equation*}
K^{[1,1]}=\frac{(u+cv)\E({Y^2})}{4\pi\D{Y}
c\sqrt{\D{T}B_1}}\E{T}{\sqrt{\D{Y}}}, \quad
K^{[1,2]}=\frac{(u+cv)\E({Y^2})}{4\pi\D{Y} c\sqrt{\D{T}B_1}}\E{Y}\sqrt{\D{T}}.
\end{equation*}

Applying Lemmas~\ref{dfgfmgmh} and~\ref{sdrtfhfgjmf} to
$\expaN{t}^{[2]}(u,c,v)$, we reduce it to the sum
\begin{equation*}
\CoR{t}^{[2]}(u,c,v) =\CoR{t}^{[2,1]}(u,c,v)+\dots+\CoR{t}^{[2,6]}(u,c,v),
\end{equation*}
where
\begin{equation*}
\begin{aligned}
%\blacktriangleright
\CoR{t}^{[2,1]}(u,c,v)&=K^{[2,1]}\int_{0}^{\frac{c(t-v)}{u+cv}}\frac{1}{1+x}
\sum_{n=\EnOne}^{\infty}n^{-3/2}\Delta^3_{n}(u+cv,x)
\\[0pt]
&\hskip
70pt\times\exp\Big\{-\tfrac{1}{2}\Big[\Lambda_{n}^2(u+cv,x)+\Delta^2_{n}(u+cv,x)\Big]\Big\}dx,
\\[0pt]
%\triangleleft\quad
\CoR{t}^{[2,2]}(u,c,v)&=-3 K^{[2,2]}
\int_{0}^{\frac{c(t-v)}{u+cv}}\frac{1}{1+x}\sum_{n=\EnOne}^{\infty}n^{-3/2}\Delta^2_{n}(u+cv,x)\Lambda_{n}(u+cv,x)
\\[0pt]
&\hskip
70pt\times\exp\Big\{-\tfrac{1}{2}\Big[\Lambda_{n}^2(u+cv,x)+\Delta^2_{n}(u+cv,x)\Big]\Big\}dx,
\\[0pt]
%\blacktriangleright
\CoR{t}^{[2,3]}(u,c,v)&=3K^{[2,3]}
\int_{0}^{\frac{c(t-v)}{u+cv}}\frac{1}{1+x}\sum_{n=\EnOne}^{\infty}n^{-3/2}\Delta_{n}(u+cv,x)\Lambda^2_{n}(u+cv,x)
\\[0pt]
&\hskip
70pt\times\exp\Big\{-\tfrac{1}{2}\Big[\Lambda_{n}^2(u+cv,x)+\Delta^2_{n}(u+cv,x)\Big]\Big\}dx,
\\[0pt]
\end{aligned}
\end{equation*}
\begin{equation*}
\begin{aligned}
%\triangleleft
\CoR{t}^{[2,4]}(u,c,v)&=-K^{[2,4]}\int_{0}^{\frac{c(t-v)}{u+cv}}
\frac{1}{1+x}\sum_{n=\EnOne}^{\infty}n^{-3/2}\Lambda^3_{n}(u+cv,x)
\\[0pt]
&\hskip
70pt\times\exp\Big\{-\tfrac{1}{2}\Big[\Lambda_{n}^2(u+cv,x)+\Delta^2_{n}(u+cv,x)\Big]\Big\}dx,
%\end{aligned}
%\end{equation*}
%\begin{equation*}
%\begin{aligned}
\\[0pt]
%\blacktriangleright
\CoR{t}^{[2,5]}(u,c,v)&=-3K^{[2,5]}
\int_{0}^{\frac{c(t-v)}{u+cv}}\frac{1}{1+x}\sum_{n=\EnOne}^{\infty}n^{-3/2}\Delta_{n}(u+cv,x)
\\[0pt]
&\hskip
70pt\times\exp\Big\{-\tfrac{1}{2}\Big[\Lambda_{n}^2(u+cv,x)+\Delta^2_{n}(u+cv,x)\Big]\Big\}dx,
\\[0pt]
%\triangleleft\quad
\CoR{t}^{[2,6]}(u,c,v)&=3K^{[2,6]}
\int_{0}^{\frac{c(t-v)}{u+cv}}\frac{1}{1+x}\sum_{n=\EnOne}^{\infty}n^{-3/2}\Lambda_{n}(u+cv,x)
\\[0pt]
&\hskip
70pt\times\exp\Big\{-\tfrac{1}{2}\Big[\Lambda_{n}^2(u+cv,x)+\Delta^2_{n}(u+cv,x)\Big]\Big\}dx,
\end{aligned}
\end{equation*}
with
\begin{equation*}
\begin{aligned}
K^{[2,1]}&=K^{[2]}\bigg(\frac{\E{T}}{\sqrt{\D{T}B_1}}\bigg)^3\frac{B_2^3}{(\E{T})^3},
&K^{[2,2]}&=K^{[2]}\bigg(\frac{\E{T}}{\sqrt{\D{T}B_1}}\bigg)^3\frac{B_2^2\sqrt{B_4}}{(\E{T})^2},
\\[0pt]
K^{[2,3]}&=K^{[2]}\bigg(\frac{\E{T}}{\sqrt{\D{T}B_1}}\bigg)^3\frac{B_2B_4}{\E{T}},
&K^{[2,4]}&=K^{[2]}\bigg(\frac{\E{T}}{\sqrt{\D{T}B_1}}\bigg)^3B_4^{3/2},
\\[0pt]
K^{[2,5]}&=K^{[2]}\frac{B_2}{\sqrt{\D{T}B_1}},\quad
&K^{[2,6]}&=K^{[2]}\frac{\E{T}\sqrt{B_4}}{\sqrt{\D{T}B_1}},
\end{aligned}
\end{equation*}
and $K^{[2]}=\dfrac{(u+cv)\E(\tilde{T}^3)\E{Y}}{12\pi c\sqrt{\D{Y}\D{T}}}$.
Applying Lemmas~\ref{dfgfmgmh} and~\ref{sdrtfhfgjmf} to
$\expaN{t}^{[3]}(u,c,v)$, we reduce it to the sum
\begin{equation*}
\CoR{t}^{[3]}(u,c,v) =\CoR{t}^{[3,1]}(u,c,v)+\dots+\CoR{t}^{[3,6]}(u,c,v),
\end{equation*}
where
\begin{equation*}
\begin{aligned}
%\blacktriangleright
\CoR{t}^{[3,1]}(u,c,v)&=-K^{[3,1]}\int_{0}^{\frac{c(t-v)}{u+cv}}
\frac{1}{1+x}\sum_{n=\EnOne}^{\infty}n^{-3/2}\Delta_{n}^3(u+cv,x)
\\[0pt]
&\hskip 70pt\times
\exp\Big\{-\tfrac{1}{2}\Big[\Lambda_{n}^2(u+cv,x)+\Delta^2_{n}(u+cv,x)\Big]\Big\}dx,
\\[0pt]
%\triangleleft
\CoR{t}^{[3,2]}(u,c,v)&=-3K^{[3,2]}\int_{0}^{\frac{c(t-v)}{u+cv}}
\frac{1}{1+x}\sum_{n=\EnOne}^{\infty}n^{-3/2}\Delta_{n}^2(u+cv,x)\Lambda_{n}(u+cv,x)
\\[0pt]
&\hskip
70pt\times\exp\Big\{-\tfrac{1}{2}\Big[\Lambda_{n}^2(u+cv,x)+\Delta^2_{n}(u+cv,x)\Big]\Big\}dx,
\\[0pt]
%\blacktriangleright
\CoR{t}^{[3,3]}(u,c,v)&=-3K^{[3,3]}\int_{0}^{\frac{c(t-v)}{u+cv}}
\frac{1}{1+x}\sum_{n=\EnOne}^{\infty}n^{-3/2}\Delta_{n}(u+cv,x)\Lambda^2_{n}(u+cv,x)
\\[0pt]
&\hskip
70pt\times\exp\Big\{-\tfrac{1}{2}\Big[\Lambda_{n}^2(u+cv,x)+\Delta^2_{n}(u+cv,x)\Big]\Big\}dx,
\\[0pt]
%\triangleleft
\CoR{t}^{[3,4]}(u,c,v)&=-K^{[3,4]}\int_{0}^{\frac{c(t-v)}{u+cv}}
\frac{1}{1+x}\sum_{n=\EnOne}^{\infty}n^{-3/2}\Lambda^3_{n}(u+cv,x)
\\[4pt]
&\hskip
70pt\times\exp\Big\{-\tfrac{1}{2}\Big[\Lambda_{n}^2(u+cv,x)+\Delta^2_{n}(u+cv,x)\Big]\Big\}dx,
%\\[0pt]
\end{aligned}
\end{equation*}
\begin{equation*}
\begin{aligned}
%\blacktriangleright
\CoR{t}^{[3,5]}(u,c,v)&=3{K^{[3,5]}}\int_{0}^{\frac{c(t-v)}{u+cv}}
\frac{1}{1+x}\sum_{n=\EnOne}^{\infty}n^{-3/2}\Delta_{n}(u+cv,x)
\\[0pt]
&\hskip
70pt\times\exp\Big\{-\tfrac{1}{2}\Big[\Lambda_{n}^2(u+cv,x)+\Delta^2_{n}(u+cv,x)\Big]\Big\}dx,
\\[0pt]
%\triangleleft
\CoR{t}^{[3,6]}(u,c,v)&=3K^{[3,6]}\int_{0}^{\frac{c(t-v)}{u+cv}}
\frac{1}{1+x}\sum_{n=\EnOne}^{\infty}n^{-3/2}\Lambda_{n}(u+cv,x)
\\[0pt]
&\hskip
70pt\times\exp\Big\{-\tfrac{1}{2}\Big[\Lambda_{n}^2(u+cv,x)+\Delta^2_{n}(u+cv,x)\Big]\Big\}dx,
\end{aligned}
\end{equation*}
with
\begin{equation*}
\begin{aligned}
{K^{[3,1]}}&=K^{[3]}\bigg(\frac{\E{Y}}{\sqrt{\D{Y}B_1}}\bigg)^3\frac{B^3_3}{(\E{Y})^3},
&K^{[3,2]}&=K^{[3]}\bigg(\frac{\E{Y}}{\sqrt{\D{Y}B_1}}\bigg)^3\frac{B^2_3\sqrt{B_4}}{(\E{Y})^2},
\\[0pt]
K^{[3,3]}&=K^{[3]}\bigg(\frac{\E{Y}}{\sqrt{\D{Y}B_1}}\bigg)^3\frac{B_3B_4}{\E{Y}},
&K^{[3,4]}&=K^{[3]}\bigg(\frac{\E{Y}}{\sqrt{\D{Y}B_1}}\bigg)^3B_4^{3/2},
\\[0pt]
K^{[3,5]}&=K^{[3]}\frac{B_3}{\sqrt{\D{Y}B_1}}, \quad
&K^{[3,6]}&=K^{[3]}\frac{\E{Y}\sqrt{B_4}}{\sqrt{\D{Y}B_1}},
\end{aligned}
\end{equation*}
and $K^{[3]}=\dfrac{(u+cv)\E({\tilde{Y}^3})\E{Y}}{12\pi c\sqrt{\D{Y}\D{T}}}$.
The rest of the proof consists in elaboration of all these summands, when
discarded are the terms of allowed order of smallness.

%%%%%%%%%%%%%%%%%%%%%%%%%%%%%%%%%%%%%%%%%%%%%%%%%%%%%%%%%%%%%%%%%%%%%%%%%%%%%%%%%%%%%%%%
\subsection{Results needed for elaboration of $\MainApprox{t}^{[1]}(u,c,v)$ and
$\CoR{t}^{[1]}(u,c,v)$--$\CoR{t}^{[3]}(u,c,v)$}\label{fghtrgfjfgj}
%%%%%%%%%%%%%%%%%%%%%%%%%%%%%%%%%%%%%%%%%%%%%%%%%%%%%%%%%%%%%%%%%%%%%%%%%%%%%%%%%%%%%%%%

Before continuing the exposition, we get together some auxiliary results.

\begin{lemma}[First decomposition of the factor $n^{-1/2}$]\label{drthyrtgfjhfg}
We have
\begin{equation*}
\frac{1}{\sqrt{n}}=\frac{\sqrt{B_4}}{\sqrt{B_1}}
\big(\Lambda_{n+1}(u+cv,x)-\Lambda_{n}(u+cv,x)\big)
%\\
+\frac{\sqrt{B_4}}{\sqrt{B_1}}\Lambda_{n+1}(u+cv,x)
\bigg(\frac{1}{2n}-\frac{1}{8n^2}+\dots\bigg).
\end{equation*}
\end{lemma}
\begin{proof}
Bearing in mind that
$1-\sqrt{1+\frac{1}{n}}=-\frac{1}{2n}+\frac{1}{8n^2}-\dots$, it is
straightforward from the identity
\begin{equation*}
\frac{1}{\sqrt{n}}=\frac{\sqrt{B_4}}{\sqrt{B_1}}\big(\Lambda_{n+1}(u+cv,x)-\Lambda_{n}(u+cv,x)\big)
-\frac{\sqrt{B_4}}{\sqrt{B_1}}\,\Lambda_{n+1}(u+cv,x)
\bigg(1-\sqrt{1+\frac{1}{n}}\,\bigg),
\end{equation*}
which is easily verified.
\end{proof}

\begin{lemma}[Second decomposition of the factor $n^{-1/2}$]\label{sdfgtrhnj}
We have
\begin{multline*}
\frac{1}{\sqrt{n}}=\frac{\sqrt{\E{Y}}}{\sqrt{(u+cv)(1+x)}}\bigg[1
-\frac{1}{2\sqrt{B_1n}}\Big(\sqrt{B_4}\Lambda_{n}(u+cv,x)
+\frac{B_3}{\E{Y}}\Delta_{n}(u+cv,x)\Big)\bigg]
\\
-\frac{\sqrt{\E{Y}}}{\sqrt{(u+cv)(1+x)}}\bigg[\frac{1}{8B_1n}\Big(\sqrt{B_4}\Lambda_{n}(u+cv,x)
+\frac{B_3}{\E{Y}}\Delta_{n}(u+cv,x)\Big)^2
\\
\times\bigg(1+\frac{\sqrt{(u+cv)(1+x)}}{\sqrt{\E{Y}}}\frac{1}{\sqrt{n}}\,\bigg)^{-1}+\dots\bigg].
\end{multline*}
\end{lemma}

\begin{proof}
The proof applies the following iterative process. We start with the identity
\begin{multline}\label{srdthrjrt}
\frac{1}{n}=\frac{\E{Y}}{(u+cv)(1+x)}-\frac{\E{Y}}{(u+cv)(1+x)\sqrt{B_1n}}
\\
\times\Big(\sqrt{B_4}\Lambda_{n}(u+cv,x)
+\frac{B_3}{\E{Y}}\Delta_{n}(u+cv,x)\Big)
\end{multline}
which is easy to verify straightforwardly. We rewrite it as
\begin{equation*}
1-\frac{(u+cv)(1+x)}{n\E{Y}}=\frac{1}{\sqrt{B_1n}}\Big(\sqrt{B_4}\Lambda_{n}(u+cv,x)
+\frac{B_3}{\E{Y}}\Delta_{n}(u+cv,x)\Big),
\end{equation*}
or
\begin{multline*}
\frac{\sqrt{(u+cv)(1+x)}}{\sqrt{n\E{Y}}}=1-\frac{1}{\sqrt{B_1n}}\Big(\sqrt{B_4}\Lambda_{n}(u+cv,x)
+\frac{B_3}{\E{Y}}\Delta_{n}(u+cv,x)\Big)
\\
\times\bigg(1+\frac{\sqrt{(u+cv)(1+x)}}{\sqrt{n\E{Y}}}\bigg)^{-1}.
\end{multline*}
It yields the following representation for $n^{-1/2}$:
\begin{multline}\label{werfge}
\frac{1}{\sqrt{n}}=\frac{\sqrt{\E{Y}}}{\sqrt{(u+cv)(1+x)}}
\bigg[1-\frac{1}{\sqrt{B_1n}}\Big(\sqrt{B_4}\Lambda_{n}(u+cv,x)
+\frac{B_3}{\E{Y}}\Delta_{n}(u+cv,x)\Big)
\\[-2pt]
\times\underbrace{\bigg(1+\frac{\sqrt{(u+cv)(1+x)}}{\sqrt{\E{Y}}}\frac{1}{\sqrt{n}}\,\bigg)^{-1}}_{\in[0,1]}\bigg].
\end{multline}
We put it in the expression marked with curly braces in the right-hand side of
\eqref{werfge}. This substitution of the expression for $n^{-1/2}$ into itself
yields the result, since
%\begin{equation*}
$\frac{1}{1+\sqrt{1-x}}=\frac12+\frac{x}{8}+\frac{x^2}{8}+\dots$,
%\end{equation*}
as $x\to 0$.
\end{proof}

%%%%%%%%%%%%%%%%%%%%%%%%%%%%%%%%%%%%%%%%%%%%%%%%%%%%%%%%%%%%%%%%%%%%%%%%%%%%%%%%%%%%%%%%
%\subsection*{Processing of exponential factor}\label{wertylykty}
%%%%%%%%%%%%%%%%%%%%%%%%%%%%%%%%%%%%%%%%%%%%%%%%%%%%%%%%%%%%%%%%%%%%%%%%%%%%%%%%%%%%%%%%

\begin{lemma}[Processing of exponential factor]\label{qwerhyjnrt}
We have
\begin{multline*}
\exp\Big\{-\tfrac{1}{2}\Delta^2_{n}(u+cv,x)\Big\}
=\exp\Bigg\{-\frac{1}{2}\Bigg(\dfrac{x-\frac{\E{T}}{\E{Y}}c(1+x)}{\frac{c\sqrt{B_1}}{(\E{Y})^{3/2}}
\sqrt{\frac{1+x}{u+cv}}}\Bigg)^2\Bigg\}
\\[4pt]
+\frac{\sqrt{B_4}}{2\sqrt{\E{Y}}\sqrt{B_1(u+cv)(1+x)}}
\Lambda_{n}(u+cv,x)\Bigg(\dfrac{x-\frac{\E{T}}{\E{Y}}c(1+x)}{\frac{c\sqrt{B_1}}{(\E{Y})^{3/2}}
\sqrt{\frac{1+x}{u+cv}}}\Bigg)^2
\exp\Bigg\{-\frac{1}{2}\Bigg(\dfrac{x-\frac{\E{T}}{\E{Y}}c(1+x)}{\frac{c\sqrt{B_1}}{(\E{Y})^{3/2}}
\sqrt{\frac{1+x}{u+cv}}}\Bigg)^2\Bigg\}
\\[4pt]
+\frac{B_3}{2\sqrt{\E{Y}}\sqrt{B_1(u+cv)(1+x)}}
\Bigg(\dfrac{x-\frac{\E{T}}{\E{Y}}c(1+x)}{\frac{c\sqrt{B_1}}{(\E{Y})^{3/2}}
\sqrt{\frac{1+x}{u+cv}}}\Bigg)^3\exp\Bigg\{-\frac{1}{2}
\Bigg(\dfrac{x-\frac{\E{T}}{\E{Y}}c(1+x)}{\frac{c\sqrt{B_1}}{(\E{Y})^{3/2}}
\sqrt{\frac{1+x}{u+cv}}}\Bigg)^2\Bigg\}+\,\cdots.
\end{multline*}
\end{lemma}

\begin{proof}
Using Taylor's formula\footnote{Plainly, it writes as
$\exp\big\{-\tfrac{1}{2}x^2\big\}=\exp\big\{-\tfrac{1}{2}x_0^2\big\}-(x-x_0)x_0
\exp\big\{-\tfrac{1}{2}x_0^2\big\}+\dots$.}, we have
\begin{multline*}
\exp\Big\{-\tfrac{1}{2}\Delta^2_{n}(u+cv,x)\Big\}
=\exp\Big\{-\tfrac{1}{2}\Delta^2_{\frac{(u+cv)(1+x)}{\E{Y}}}(u+cv,x)\Big\}
\\
-\Big(\Delta_{n}(u+cv,x)-\Delta_{\frac{(u+cv)(1+x)}{\E{Y}}}(u+cv,x)\Big)
\Delta_{\frac{(u+cv)(1+x)}{\E{Y}}}(u+cv,x)
\\
\times\exp\Big\{-\tfrac{1}{2}\Delta^2_{\frac{(u+cv)(1+x)}{\E{Y}}}(u+cv,x)\Big\}+\dots,
\end{multline*}
where by definition of $\Delta_{n}(u+cv,x)$ and equation \eqref{werfge}, we
have
\begin{multline}\label{ftyjhgmh}
\Delta_{n}(u+cv,x)-\Delta_{\frac{(u+cv)(1+x)}{\E{Y}}}(u+cv,x)=
-\Delta_{\frac{(u+cv)(1+x)}{\E{Y}}}(u+cv,x)\frac{1}{\sqrt{B_1n}}
\\
\times\big(\sqrt{B_4}\Lambda_{n}(u+cv,x)
+\frac{B_3}{\E{Y}}\Delta_{n}(u+cv,x)\big)
\bigg(1+\sqrt{\frac{(u+cv)(1+x)}{\E{Y}}}\frac{1}{\sqrt{n}}\,\bigg)^{-1}.
\end{multline}
That yields
\begin{multline*}
\exp\Big\{-\tfrac{1}{2}\Delta^2_{n}(u+cv,x)\Big\}
-\exp\Big\{-\tfrac{1}{2}\Delta^2_{\frac{(u+cv)(1+x)}{\E{Y}}}(u+cv,x)\Big\}
\\
%=\frac{1}{2\sqrt{B_1n}}\Delta^2_{\frac{(u+cv)(1+x)}{\E{Y}}}(u+cv,x)
%\Big(\sqrt{B_4}\Lambda_{n}(u+cv,x) +\frac{B_3}{\E{Y}}\Delta_{n}(u+cv,x)\Big)
%\\
%\times\exp\Big\{-\tfrac{1}{2}\Delta^2_{\frac{(u+cv)(1+x)}{\E{Y}}}(u+cv,x)\Big\}+\dots
%\\
%=\frac{\sqrt{B_4}}{2\sqrt{B_1n}}\Lambda_{n}(u+cv,x)\Delta^2_{\frac{(u+cv)(1+x)}{\E{Y}}}(u+cv,x)
%\exp\Big\{-\tfrac{1}{2}\Delta^2_{\frac{(u+cv)(1+x)}{\E{Y}}}(u+cv,x)\Big\}
%\\
%+\frac{B_3}{2\E{Y}\sqrt{B_1n}}\Delta^3_{\frac{(u+cv)(1+x)}{\E{Y}}}(u+cv,x)
%\exp\Big\{-\tfrac{1}{2}\Delta^2_{\frac{(u+cv)(1+x)}{\E{Y}}}(u+cv,x)\Big\}+\dots
%\\
=\frac{\sqrt{B_4}}{2\sqrt{\E{Y}}\sqrt{B_1(u+cv)(1+x)}}\Lambda_{n}(u+cv,x)
\Delta^2_{\frac{(u+cv)(1+x)}{\E{Y}}}(u+cv,x)
\\
\times\exp\Big\{-\tfrac{1}{2}\Delta^2_{\frac{(u+cv)(1+x)}{\E{Y}}}(u+cv,x)\Big\}
\\
+\frac{B_3}{2\sqrt{\E{Y}}\sqrt{B_1(u+cv)(1+x)}}\Delta^3_{\frac{(u+cv)(1+x)}{\E{Y}}}(u+cv,x)
\\
\times\exp\Big\{-\tfrac{1}{2}\Delta^2_{\frac{(u+cv)(1+x)}{\E{Y}}}(u+cv,x)\Big\}+\dots,
\end{multline*}
as required.
\end{proof}

%%%%%%%%%%%%%%%%%%%%%%%%%%%%%%%%%%%%%%%%%%%%%%%%%%%%%%%%%%%%%%%%%%%%%%%%%%%%%%%%%%%%%%%%
\subsection{Elaboration of $\MainApprox{t}^{[1]}(u,c,v)$}\label{ergwerghwergh}
%%%%%%%%%%%%%%%%%%%%%%%%%%%%%%%%%%%%%%%%%%%%%%%%%%%%%%%%%%%%%%%%%%%%%%%%%%%%%%%%%%%%%%%%

Let us formulate the main result of this section\footnote{We can easily prove
that the remainder term in Lemma~\ref{wertgrherw} is of order
$\underline{O}((u+cv)^{-2})$. But it is not essential since
Lemma~\ref{wertgrherw}, as well as Lemmas~\ref{tyujrtjjk}--\ref{wrthyjkt}
formulated below are used as components in the proof of the fundamental
Theorem~\ref{srdthjrf}. The rate
$\underline{O}\big(\frac{\ln(u+cv)}{(u+cv)^2}\big)$ in this theorem is due to
our estimation of the remainder terms, which details are given in
\citeNP{[Malinovskii 2017]}.}.

\begin{lemma}\label{wertgrherw}
We have
\begin{equation*}
\sup_{t>v}\Big|\,\MainApprox{t}^{[1]}(u,c,v)
-\Int{t}{M}(u,c,v)+\frac{\E{T}\D{Y}}{2cD^2(\E{Y})^2}(\Int{t}{F}(u,c,v)-
\Int{t}{S}(u,c,v))\Big|=\underline{O}\bigg(\frac{\ln(u+cv)}{(u+cv)^2}\bigg),
\end{equation*}
as $u+cv\to\infty$.
\end{lemma}

\begin{remark}[Notation agreement]
For brevity, we will use simplified notation with sign $\Rightarrow$ for the
approximation, like
$\MainApprox{t}^{[1]}(u,c,v)\Rightarrow\Int{t}{M}(u,c,v)-\frac{\E{T}\D{Y}}{2cD^2(\E{Y})^2}(\Int{t}{F}(u,c,v)
-\Int{t}{S}(u,c,v))$ in the statement of Lemma~\ref{wertgrherw}.
\end{remark}

\begin{proof}
First stage of the proof consists in processing the factor
$\exp\big\{-\tfrac{1}{2}\Delta^2_{n}(u+cv,x)\big\}$ in
$\MainApprox{t}^{[1]}(u,c,v)$ by means of Lemma~\ref{qwerhyjnrt}. It is easy to
verify that holding the required accuracy, we have
$\MainApprox{t}^{[1]}(u,c,v)$ approximated by the sum
\begin{equation*}
\MainApprox{t}^{[1,1]}(u,c,v)+\MainApprox{t}^{[1,2]}(u,c,v)
+\MainApprox{t}^{[1,3]}(u,c,v),
\end{equation*}
where
\begin{equation*}
\begin{aligned}
%\bigodot
\MainApprox{t}^{[1,1]}(u,c,v)&=\frac{(u+cv)\E{Y}}{2\pi
c\sqrt{\D{T}\D{Y}}}\int_{0}^{\frac{c(t-v)}{u+cv}} \frac{1}{1+x}
\\[0pt]
&\times\exp\Bigg\{-\frac{1}{2}\Bigg(\dfrac{x-\frac{\E{T}}{\E{Y}}c(1+x)}{\frac{c\sqrt{B_1}}{(\E{Y})^{3/2}}
\sqrt{\frac{1+x}{u+cv}}}\Bigg)^2\Bigg\}\sum_{n=\EnOne}^{\infty}n^{-1}\exp\Big\{-\tfrac{1}{2}
\Lambda_{n}^2(u+cv,x)\Big\}dx,
\\[0pt]
%\vartriangleright
\MainApprox{t}^{[1,2]}(u,c,v)&=\frac{(u+cv)\E{Y}}{2\pi
c\sqrt{\D{T}\D{Y}}}\frac{\sqrt{B_4}}{2\sqrt{\E{Y}}\sqrt{B_1(u+cv)}}\int_{0}^{\frac{c(t-v)}{u+cv}}
\frac{1}{(1+x)^{3/2}}
\\[0pt]
&\times\Bigg(\dfrac{x-\frac{\E{T}}{\E{Y}}c(1+x)}{\frac{c\sqrt{B_1}}{(\E{Y})^{3/2}}
\sqrt{\frac{1+x}{u+cv}}}\Bigg)^2\exp\Bigg\{-\frac{1}{2}\Bigg(\dfrac{x-\frac{\E{T}}{\E{Y}}c(1+x)}{\frac{c\sqrt{B_1}}{(\E{Y})^{3/2}}
\sqrt{\frac{1+x}{u+cv}}}\Bigg)^2\Bigg\}
\\[0pt]
&\times\sum_{n=\EnOne}^{\infty}n^{-1}\Lambda_{n}(u+cv,x)
\exp\Big\{-\tfrac{1}{2}\Lambda_{n}^2(u+cv,x)\Big\}dx,
%\\[0pt]
\end{aligned}
\end{equation*}
\begin{equation*}
\begin{aligned}
%\blacktriangleright
\MainApprox{t}^{[1,3]}(u,c,v)&=\frac{(u+cv)\E{Y}}{2\pi
c\sqrt{\D{T}\D{Y}}}\frac{B_3}{2\sqrt{\E{Y}}\sqrt{B_1(u+cv)}}\int_{0}^{\frac{c(t-v)}{u+cv}}
\frac{1}{(1+x)^{3/2}}
\Bigg(\dfrac{x-\frac{\E{T}}{\E{Y}}c(1+x)}{\frac{c\sqrt{B_1}}{(\E{Y})^{3/2}}
\sqrt{\frac{1+x}{u+cv}}}\Bigg)^3
\\[0pt]
&\times\exp\Bigg\{-\frac{1}{2}\Bigg(\dfrac{x-\frac{\E{T}}{\E{Y}}c(1+x)}{\frac{c\sqrt{B_1}}{(\E{Y})^{3/2}}
\sqrt{\frac{1+x}{u+cv}}}\Bigg)^2\Bigg\}
\sum_{n=\EnOne}^{\infty}n^{-1}\exp\Big\{-\tfrac{1}{2}
\Lambda_{n}^2(u+cv,x)\Big\}dx.
\end{aligned}
\end{equation*}
It is noteworthy, using identity \eqref{srdthrjrt} and Lemma~\ref{qaerfgwehg},
that
\begin{equation}\label{srdtghnf}
\begin{aligned}
\sum_{n=\EnOne}^{\infty}&n^{-1}\Lambda_{n}(u+cv,x)\exp\Big\{-\tfrac{1}{2}\Lambda_{n}^2(u+cv,x)\Big\}
\Rightarrow 0,
\\
\sum_{n=\EnOne}^{\infty}&n^{-1}\exp\Big\{-\tfrac{1}{2}\Lambda_{n}^2(u+cv,x)\Big\}
\Rightarrow\sqrt{2\pi}
\frac{\sqrt{\E{Y}}}{\sqrt{(u+cv)(1+x)}}\frac{\sqrt{B_4}}{\sqrt{B_1}},
\end{aligned}
\end{equation}
and that
\begin{multline}\label{qwerfgwer}
\int_{0}^{\frac{c(t-v)}{u+cv}}\frac{1}{(1+x)^2}
\Bigg(\dfrac{x-\frac{\E{T}}{\E{Y}}c(1+x)}{\frac{c\sqrt{B_1}}{(\E{Y})^{3/2}}
\sqrt{\frac{1+x}{u+cv}}}\Bigg)^3\underbrace{\frac{1}{\sqrt{2\pi}}\exp\Bigg\{-\frac{1}{2}
\Bigg(\dfrac{x-\frac{\E{T}}{\E{Y}}c(1+x)}{\frac{c\sqrt{B_1}}{(\E{Y})^{3/2}}
\sqrt{\frac{1+x}{u+cv}}}\Bigg)^2\Bigg\}}_{\sqrt{\frac{c^2
D^2(1+x)}{u+cv}}\Ugauss{cM(1+x)}{\frac{c^2D^2(1+x)}{u+cv}}(x)}dx
%\\
%=\int_{0}^{\frac{c(t-v)}{u+cv}}\frac{1}{(1+x)^2}
%\dfrac{\Big(x-\frac{\E{T}}{\E{Y}}c(1+x)\Big)^3}{\Big(\frac{c\sqrt{B_1}}{(\E{Y})^{3/2}}
%\sqrt{\frac{1+x}{u+cv}}\,\Big)^3}\sqrt{\frac{c^2
%D^2(1+x)}{u+cv}}\Ugauss{cM(1+x)}{\frac{c^2D^2(1+x)}{u+cv}(x)}dx
\\[0pt]
=\frac{(u+cv)}{(cD)^2}\int_{0}^{\frac{c(t-v)}{u+cv}}
\dfrac{(x-\frac{\E{T}}{\E{Y}}c(1+x))^3}{(1+x)^3}\Ugauss{cM(1+x)}{\frac{c^2D^2(1+x)}{u+cv}(x)}dx
=\Int{t}{S}(u,c,v).
\end{multline}

We have
\begin{equation}\label{ewrth4y5hjrtj}
\begin{aligned}
&\MainApprox{t}^{[1,1]}(u,c,v)\Rightarrow\text{see next stage},
%\end{equation*}
\\[0pt]
%\begin{equation*}
&\MainApprox{t}^{[1,2]}(u,c,v)\Rightarrow 0,
%\end{equation*}
\\[0pt]
%\begin{equation*}
&\MainApprox{t}^{[1,3]}(u,c,v)\Rightarrow
\frac{}{}\frac{B_3\E{Y}\sqrt{B_4}}{2c\sqrt{\D{T}\D{Y}}B_1}\,\Int{t}{S}(u,c,v)
\eqOK\frac{\E{T}\D{Y}}{2cD^2(\E{Y})^2}\,\Int{t}{S}(u,c,v).
\end{aligned}
\end{equation}

Second stage is transformation of the factor $n^{1/2}$ in the summand
$\MainApprox{t}^{[1,1]}(u,c,v)$ by use of Lemma~\ref{sdfgtrhnj}. We have
\begin{equation*}
\MainApprox{t}^{[1,1]}(u,c,v)=\MainApprox{t}^{[1,1,1]}(u,c,v)+\MainApprox{t}^{[1,1,2]}(u,c,v)
+\MainApprox{t}^{[1,1,3]}(u,c,v)+\dots,
\end{equation*}
where
\begin{equation*}
\begin{aligned}
\MainApprox{t}^{[1,1,1]}(u,c,v)&=\frac{(u+cv)\E{Y}}{2\pi
c\sqrt{\D{T}\D{Y}}}\frac{\sqrt{\E{Y}}}{\sqrt{(u+cv)}}\int_{0}^{\frac{c(t-v)}{u+cv}}
\frac{1}{(1+x)^{3/2}}
\\[0pt]
&\times
\exp\Bigg\{-\frac{1}{2}\Bigg(\dfrac{x-\frac{\E{T}}{\E{Y}}c(1+x)}{\frac{c\sqrt{B_1}}{(\E{Y})^{3/2}}
\sqrt{\frac{1+x}{u+cv}}}\Bigg)^2\Bigg\} \sum_{n=\EnOne}^{\infty}n^{-1/2}
%\\[0pt]
%&\times
\exp\Big\{-\tfrac{1}{2}\Lambda_{n}^2(u+cv,x)\Big\}dx,
\\[0pt]
%\vartriangleright
\MainApprox{t}^{[1,1,2]}(u,c,v)&=-\frac{\sqrt{B_4}}{2\sqrt{B_1}}\frac{(u+cv)\E{Y}}{2\pi
c\sqrt{\D{T}\D{Y}}}\frac{\sqrt{\E{Y}}}{\sqrt{(u+cv)}}\int_{0}^{\frac{c(t-v)}{u+cv}}
\frac{1}{(1+x)^{3/2}}
\\[0pt]
&\times\exp\Bigg\{-\frac{1}{2}\Bigg(\dfrac{x-\frac{\E{T}}{\E{Y}}c(1+x)}
{\frac{c\sqrt{B_1}}{(\E{Y})^{3/2}} \sqrt{\frac{1+x}{u+cv}}}\Bigg)^2\Bigg\}
\sum_{n=\EnOne}^{\infty}n^{-1}\Lambda_{n}(u+cv,x)
\\[0pt]
&\times \exp\Big\{-\tfrac{1}{2}\Lambda_{n}^2(u+cv,x)\Big\}dx,
%\\
\end{aligned}
\end{equation*}
\begin{equation*}
\begin{aligned}
%\blacktriangleright
\MainApprox{t}^{[1,1,3]}(u,c,v)&=-\frac{1}{2\sqrt{B_1}}\frac{B_3}{\E{Y}}\frac{(u+cv)\E{Y}}{2\pi
c\sqrt{\D{T}\D{Y}}}\frac{\sqrt{\E{Y}}}{\sqrt{(u+cv)}}\int_{0}^{\frac{c(t-v)}{u+cv}}
\frac{1}{(1+x)^{3/2}}
\\[0pt]
&\times\Bigg(\dfrac{x-\frac{\E{T}}{\E{Y}}c(1+x)}{\frac{c\sqrt{B_1}}{(\E{Y})^{3/2}}
\sqrt{\frac{1+x}{u+cv}}}\Bigg)\exp\Bigg\{-\frac{1}{2}\Bigg(\dfrac{x-\frac{\E{T}}{\E{Y}}c(1+x)}{\frac{c\sqrt{B_1}}{(\E{Y})^{3/2}}
\sqrt{\frac{1+x}{u+cv}}}\Bigg)^2\Bigg\}
\\[0pt]
&\times\sum_{n=\EnOne}^{\infty}n^{-1}\exp\Big\{-\tfrac{1}{2}\Lambda_{n}^2(u+cv,x)\Big\}dx.
\end{aligned}
\end{equation*}

Bearing in mind \eqref{srdtghnf}, it is noteworthy that
\begin{multline}\label{wqergthrt}
\int_{0}^{\frac{c(t-v)}{u+cv}}
\frac{1}{(1+x)^{2}}\Bigg(\dfrac{x-\frac{\E{T}}{\E{Y}}c(1+x)}{\frac{c\sqrt{B_1}}{(\E{Y})^{3/2}}
\sqrt{\frac{1+x}{u+cv}}}\Bigg)\underbrace{\frac{1}{\sqrt{2\pi}}\exp\Bigg\{-\frac{1}{2}\Bigg(\dfrac{x-\frac{\E{T}}{\E{Y}}c(1+x)}{\frac{c\sqrt{B_1}}{(\E{Y})^{3/2}}
\sqrt{\frac{1+x}{u+cv}}}\Bigg)^2\Bigg\}}_{\sqrt{\frac{c^2
D^2(1+x)}{u+cv}}\Ugauss{cM(1+x)}{\frac{c^2D^2(1+x)}{u+cv}}(x)}dx
\\[-6pt]
=\int_{0}^{\frac{c(t-v)}{u+cv}}
\dfrac{x-\frac{\E{T}}{\E{Y}}c(1+x)}{(1+x)^{2}}\Ugauss{cM(1+x)}{\frac{c^2D^2(1+x)}{u+cv}}(x)dx
\eqOK\Int{t}{F}(u,c,v).
\end{multline}

We have
\begin{equation}\label{wertgrejh}
\begin{aligned}
&\MainApprox{t}^{[1,1,1]}(u,c,v)\Rightarrow\text{see next stage},
\\[0pt]
&\MainApprox{t}^{[1,1,2]}(u,c,v)\Rightarrow 0,
\\[0pt]
&\MainApprox{t}^{[1,1,3]}(u,c,v)=
-\frac{B_3\E{Y}}{2cB_1}\,\Int{t}{F}(u,c,v)\eqOK-\frac{\E{T}\D{Y}}{2cD^2(\E{Y})^2}\,\Int{t}{F}(u,c,v).
\end{aligned}
\end{equation}

Third stage is transformation of the factor $n^{1/2}$ in the summand
$\MainApprox{t}^{[1,1,1]}(u,c,v)$ by use of Lemma~\ref{drthyrtgfjhfg}. We have
\begin{equation*}
\MainApprox{t}^{[1,1,1]}(u,c,v)=\MainApprox{t}^{[1,1,1,1]}(u,c,v)+\MainApprox{t}^{[1,1,1,2]}(u,c,v)+\dots,
\end{equation*}
where
\begin{equation*}
\begin{aligned}
\MainApprox{t}^{[1,1,1,1]}(u,c,v)&=\frac{\sqrt{B_4}}{\sqrt{B_1}}\frac{(u+cv)\E{Y}}{2\pi
c\sqrt{\D{T}\D{Y}}}\frac{\sqrt{\E{Y}}}{\sqrt{(u+cv)}}\int_{0}^{\frac{c(t-v)}{u+cv}}
\frac{1}{(1+x)^{3/2}}
\\[0pt]
&\times\exp\Bigg\{-\frac{1}{2}\Bigg(\dfrac{x-\frac{\E{T}}{\E{Y}}c(1+x)}{\frac{c\sqrt{B_1}}{(\E{Y})^{3/2}}
\sqrt{\frac{1+x}{u+cv}}}\Bigg)^2\Bigg\}
\\[-2pt]
&\times\sum_{n=\EnOne}^{\infty}
\exp\Big\{-\tfrac{1}{2}\Lambda_{n}^2(u+cv,x)\Big\}
\big(\Lambda_{n+1}(u+cv,x)-\Lambda_{n}(u+cv,x)\big)dx,
%\\[0pt]
\end{aligned}
\end{equation*}
\begin{equation*}
\begin{aligned}
%\vartriangleright
\MainApprox{t}^{[1,1,1,2]}(u,c,v)&=\frac{1}{2}\frac{\sqrt{B_4}}{\sqrt{B_1}}\frac{(u+cv)\E{Y}}{2\pi
c\sqrt{\D{T}\D{Y}}}\frac{\sqrt{\E{Y}}}{\sqrt{(u+cv)}}\int_{0}^{\frac{c(t-v)}{u+cv}}
\frac{1}{(1+x)^{3/2}}
\\[0pt]
&\times\exp\Bigg\{-\frac{1}{2}\Bigg(\dfrac{x-\frac{\E{T}}{\E{Y}}c(1+x)}{\frac{c\sqrt{B_1}}{(\E{Y})^{3/2}}
\sqrt{\frac{1+x}{u+cv}}}\Bigg)^2\Bigg\}
\\[0pt]
&\times\sum_{n=\EnOne}^{\infty}
\bigg(\frac{1}{n}-\frac{1}{4n^2}+\dots\bigg)\Lambda_{n+1}(u+cv,x)
\exp\Big\{-\tfrac{1}{2}\Lambda_{n}^2(u+cv,x)\Big\}dx.
\end{aligned}
\end{equation*}

We have
\begin{equation}\label{wretrjr}
\begin{aligned}
&\MainApprox{t}^{[1,1,1,1]}(u,c,v)\Rightarrow\text{see next stage},
\\[0pt]
&\MainApprox{t}^{[1,1,2,1]}(u,c,v)\Rightarrow 0.
\end{aligned}
\end{equation}

Fourth stage is approximation of integral sum in
$\MainApprox{t}^{[1,1,1,1]}(u,c,v)$ by use of Lemma~\ref{qaerfgwehg} applied to
the factor
\begin{equation*}
\begin{aligned}
\sum_{n=\EnOne}^{\infty}
\exp\Big\{-\tfrac{1}{2}\Lambda_{n}^2(u+cv,x)\Big\}\big(\Lambda_{n+1}(u+cv,x)-\Lambda_{n}(u+cv,x)\big)dx
\end{aligned}
\end{equation*}
in the expression for $\MainApprox{t}^{[1,1,1,1]}(u,c,v)$. We bear in mind that
\begin{multline*}
\int_{0}^{\frac{c(t-v)}{u+cv}}
\frac{1}{(1+x)^{3/2}}\underbrace{\frac{1}{\sqrt{2\pi}}\exp\Bigg\{-\frac{1}{2}\Bigg(\dfrac{x-\frac{\E{T}}{\E{Y}}c(1+x)}{\frac{c\sqrt{B_1}}{(\E{Y})^{3/2}}
\sqrt{\frac{1+x}{u+cv}}}\Bigg)^2\Bigg\}}_{\sqrt{\frac{c^2
D^2(1+x)}{u+cv}}\Ugauss{cM(1+x)}{\frac{c^2D^2(1+x)}{u+cv}}(x)}dx
\\[-4pt]
=\frac{c
D}{\sqrt{u+cv}}\int_{0}^{\frac{c(t-v)}{u+cv}}\frac{1}{(1+x)}\Ugauss{cM(1+x)}{\frac{c^2D^2(1+x)}{u+cv}}(x)dx
\eqOK\frac{c D}{\sqrt{u+cv}}\Int{t}{M}(u,c,v).
\end{multline*}

%\begin{assertion}\label{ertyhrtjr}
We have
%\marginpar{$B_4=\D{Y}\D{T}$, $\frac{c\sqrt{B_1}}{(\E{Y})^{3/2}}=cD$}
%\begin{equation*}
%\MainApprox{t}^{[1,1,1,1]}(u,c,v)\Rightarrow\frac{\sqrt{B_4}}{\sqrt{B_1}}\frac{(u+cv)\E{Y}}{
%c\sqrt{\D{T}\D{Y}}}\frac{\sqrt{\E{Y}}}{\sqrt{(u+cv)}}\frac{c
%D}{\sqrt{u+cv}}\Int{t}{M}(u,c,v)
%\end{equation*}
%\begin{equation*}
$\MainApprox{t}^{[1,1,1,1]}(u,c,v)\Rightarrow\Int{t}{M}(u,c,v)$.
%\end{equation*}
%\end{assertion}
Together with \eqref{ewrth4y5hjrtj}, \eqref{wertgrejh}, \eqref{wretrjr}, it
gives the desired result.
\end{proof}

%%%%%%%%%%%%%%%%%%%%%%%%%%%%%%%%%%%%%%%%%%%%%%%%%%%%%%%%%%%%%%%%%%%%%%%%%%%%%%%%%%%%%%%%
\subsection{Elaboration of $\CoR{t}^{[1]}(u,c,v)$}\label{tyujetijer}
%%%%%%%%%%%%%%%%%%%%%%%%%%%%%%%%%%%%%%%%%%%%%%%%%%%%%%%%%%%%%%%%%%%%%%%%%%%%%%%%%%%%%%%%

%\begin{multline*}
%\CoR{t}^{[1,1]}(u,c,v) =\frac{\E{T}\E{Y}\D{Y}}{2cB_1}
%\Int{t}{F}(u,c,v)+\frac{\E{T}(\E{Y})^3}{2cB_1} \Int{t}{F}(u,c,v)
%\\
%=\frac{\E{T}\D{Y}}{2cD^2(\E{Y})^2}\Int{t}{F}(u,c,v)+\frac{\E{T}}{2cD^2}
%\Int{t}{F}(u,c,v)=\frac{\E{T}\E({Y^2})}{2cD^2(\E{Y})^2}\Int{t}{F}(u,c,v)
%\end{multline*}

\begin{lemma}\label{tyujrtjjk}
We have
\begin{equation*}
\sup_{t>v}\Big|\,\CoR{t}^{[1]}(u,c,v)
-\Big(\frac{\E{T}\D{Y}}{2cD^2(\E{Y})^2}+\frac{\E{T}}{2cD^2}\Big)\Int{t}{F}(u,c,v)\Big|
=\underline{O}\bigg(\frac{\ln(u+cv)}{(u+cv)^2}\bigg),
\end{equation*}
as $u+cv\to\infty$.
\end{lemma}

\begin{proof}
We start with equation \eqref{2345ty45uj4r}. Using Lemma~\ref{qwerhyjnrt} for
processing of exponential factor, equation \eqref{ftyjhgmh}, identity
\eqref{srdthrjrt} for processing the factor $n^{-1}$, Lemma~\ref{drthyrtgfjhfg}
to switch to the integral sum and Lemma~\ref{qaerfgwehg} to approximate it by
respective integral, and bearing in mind \eqref{wqergthrt}, and arguing the
same as above, we have
\begin{equation*}
\begin{aligned}
\CoR{t}^{[1,1]}(u,c,v)&=K^{[1,1]}\frac{\E{Y}}{(u+cv)}\frac{\sqrt{B_4}}{\sqrt{B_1}}2\pi\int_{0}^{\frac{c(t-v)}{u+cv}}
\frac{1}{(1+x)^2}\Bigg(\dfrac{x-\frac{\E{T}}{\E{Y}}c(1+x)}{\frac{c\sqrt{B_1}}{(\E{Y})^{3/2}}
\sqrt{\frac{1+x}{u+cv}}}\Bigg)
\\
&\times\frac{1}{\sqrt{2\pi}}\exp\Bigg\{-\frac{1}{2}\Bigg(\dfrac{x-\frac{\E{T}}{\E{Y}}c(1+x)}{\frac{c\sqrt{B_1}}{(\E{Y})^{3/2}}
\sqrt{\frac{1+x}{u+cv}}}\Bigg)^2\Bigg\}dx
\\
&=2\pi
K^{[1,1]}\frac{\E{Y}}{(u+cv)}\frac{\sqrt{B_4}}{\sqrt{B_1}}\Int{t}{F}(u,c,v)\eqOK\frac{\E{Y}\E{T}\E({Y^2})}{2cB_1}\Int{t}{F}(u,c,v).
\end{aligned}
\end{equation*}

%It is easily seen that
%\begin{equation*}
%\begin{aligned}
%2\pi
%K^{[1,1]}\frac{\E{Y}}{(u+cv)}\frac{\sqrt{B_4}}{\sqrt{B_1}}
%\eqOK\frac{\E{Y}\E{T}\E({Y^2})}{2cB_1}\Int{t}{F}(u,c,v).
%\end{aligned}
%\end{equation*}

Similar investigation of $\CoR{t}^{[1,2]}(u,c,v)$ yields the following result.
We have
\begin{equation*}
\begin{aligned}
&\CoR{t}^{[1,1]}(u,c,v)\Rightarrow
\frac{\E{Y}\E{T}\E({Y^2})}{2cB_1}\Int{t}{F}(u,c,v)\eqOK\bigg(\frac{\E{T}\D{Y}}{2cD^2(\E{Y})^2}
+\frac{\E{T}}{2cD^2}\bigg)\Int{t}{F}(u,c,v),
\\[0pt]
&\CoR{t}^{[1,2]}(u,c,v)\Rightarrow 0.
\end{aligned}
\end{equation*}
The proof is complete.
\end{proof}

%%%%%%%%%%%%%%%%%%%%%%%%%%%%%%%%%%%%%%%%%%%%%%%%%%%%%%%%%%%%%%%%%%%%%%%%%%%%%%%%%%%%%%%%
\subsection{Elaboration of $\CoR{t}^{[2]}(u,c,v)$}\label{ertgrehr}
%%%%%%%%%%%%%%%%%%%%%%%%%%%%%%%%%%%%%%%%%%%%%%%%%%%%%%%%%%%%%%%%%%%%%%%%%%%%%%%%%%%%%%%%

\begin{lemma}\label{rt56ytrujrt}
We have
\begin{multline*}
\sup_{t>v}\bigg|\,\CoR{t}^{[2]}(u,c,v)-\frac{\E(T-\E{T})^3}{6cD^4\E{Y}}\Int{t}{S}(u,c,v)
\\
-\frac{\E(T-\E{T})^3}{2cD^2\D{T}}\bigg(\dfrac{(\E{T})^2\D{Y}}{D^2(\E{Y})^3}-1\bigg)\Int{t}{F}(u,c,v)\bigg|
=\underline{O}\bigg(\frac{\ln(u+cv)}{(u+cv)^2}\bigg),
\end{multline*}
as $u+cv\to\infty$.
\end{lemma}

\begin{proof}
Arguing the same as above, we have
\begin{equation*}
\begin{aligned}
&\CoR{t}^{[2,1]}(u,c,v)\Rightarrow
K^{[2,1]}\frac{2\pi\E{Y}}{(u+cv)}\frac{\sqrt{B_4}}{\sqrt{B_1}}
\Int{t}{S}(u,c,v)\eqOK\frac{\E(T-\E{T})^3}{6cD^4\E{Y}}\Int{t}{S}(u,c,v),
\\[0pt]
&\CoR{t}^{[2,2]}(u,c,v)\Rightarrow 0,
\\[0pt]
&\CoR{t}^{[2,3]}(u,c,v)\Rightarrow
3K^{[2,3]}\frac{2\pi\E{Y}}{(u+cv)}\frac{\sqrt{B_4}}{\sqrt{B_1}}
\Int{t}{F}(u,c,v)\eqOK\dfrac{(\E{T})^2\D{Y}\E(T-\E{T})^3}{2
cD^4(\E{Y})^3\D{T}}\Int{t}{F}(u,c,v),
\\[0pt]
&\CoR{t}^{[2,4]}(u,c,v)\Rightarrow 0,
\\[0pt]
&\CoR{t}^{[2,5]}(u,c,v)\Rightarrow
-3K^{[2,5]}\frac{2\pi\E{Y}}{(u+cv)}\frac{\sqrt{B_4}}{\sqrt{B_1}}
\Int{t}{F}(u,c,v)\eqOK-\dfrac{\E(T-\E{T})^3}{2 cD^2\D{T}}\Int{t}{F}(u,c,v),
\\[0pt]
&\CoR{t}^{[2,6]}(u,c,v)\Rightarrow 0.
\end{aligned}
\end{equation*}
It gives the desired result.
\end{proof}

%%%%%%%%%%%%%%%%%%%%%%%%%%%%%%%%%%%%%%%%%%%%%%%%%%%%%%%%%%%%%%%%%%%%%%%%%%%%%%%%%%%%%%%%
\subsection{Elaboration of $\CoR{t}^{[3]}(u,c,v)$}\label{sdfghjtyk}
%%%%%%%%%%%%%%%%%%%%%%%%%%%%%%%%%%%%%%%%%%%%%%%%%%%%%%%%%%%%%%%%%%%%%%%%%%%%%%%%%%%%%%%%

\begin{lemma}\label{wrthyjkt}
We have
\begin{multline*}
\sup_{t>v}\bigg|\,\CoR{t}^{[3]}(u,c,v)+\dfrac{(\E{T})^3\E(Y-\E{Y})^3}{6cD^4(\E{Y})^4}\Int{t}{S}(u,c,v)
\\
+\frac{\E{T}\E(Y-\E{Y})^3}{2cD^2\E{Y}\D{Y}}\bigg(\dfrac{\D{T}}{D^2\E{Y}}
-1\bigg)\Int{t}{F}(u,c,v)\bigg|
=\underline{O}\bigg(\frac{\ln(u+cv)}{(u+cv)^2}\bigg),
\end{multline*}
as $u+cv\to\infty$.
\end{lemma}

\begin{proof}
Arguing the same as above, we have
\begin{equation*}
\begin{aligned}
&\CoR{t}^{[3,1]}(u,c,v)\Rightarrow
-K^{[3,1]}\frac{2\pi\E{Y}}{(u+cv)}\frac{\sqrt{B_4}}{\sqrt{B_1}}\Int{t}{S}(u,c,v)
\eqOK-\dfrac{(\E{T})^3\E(Y-\E{Y})^3}{6cD^4(\E{Y})^4}\Int{t}{S}(u,c,v),
\\[0pt]
&\CoR{t}^{[3,2]}(u,c,v)\Rightarrow 0,
\\[0pt]
&\CoR{t}^{[3,3]}(u,c,v)\Rightarrow
-3K^{[3,3]}\frac{2\pi\E{Y}}{(u+cv)}\frac{\sqrt{B_4}}{\sqrt{B_1}}\Int{t}{F}(u,c,v)
\eqOK-\dfrac{\E{T}\D{T}\E(Y-\E{Y})^3}{2cD^4\D{Y}(\E{Y})^2}\Int{t}{F}(u,c,v),
\\[0pt]
&\CoR{t}^{[3,4]}(u,c,v)\Rightarrow 0,
\\[0pt]
&\CoR{t}^{[3,5]}(u,c,v)\Rightarrow
3K^{[3,5]}\frac{2\pi\E{Y}}{(u+cv)}\frac{\sqrt{B_4}}{\sqrt{B_1}}
\Int{t}{F}(u,c,v)\eqOK\dfrac{\E{T}\E(Y-\E{Y})^3}{2cD^2\E{Y}\D{Y}}\Int{t}{F}(u,c,v),
\\[0pt]
&\CoR{t}^{[3,6]}(u,c,v)\Rightarrow 0.
\end{aligned}
\end{equation*}
It gives the desired result.
\end{proof}

The proof of Theorem~\ref{srdthjrf} follows from collecting the results of
Lemmas~\ref{wertgrherw}--\ref{wrthyjkt}, and is complete.

\begin{proof}[Proof of Theorem~\ref{dthyjrt}]
%For $M(s)=\inf\{k\geqslant 1:\sum_{i=1}^{k}Y_{i}>s\}-1$ and $0<v<t$, we have
%(see equation \eqref{qwretgjhmg})
%\begin{equation}\label{qwretgjhmg}
%\P\{v<\timeR\leqslant
%t\mid\T{1}=v\}=\int_{v}^{t}\dfrac{u+cv}{u+cz}\sum_{n=1}^{\infty}
%\P\big\{M(u+cz)=n\big\}f_{T}^{*n}(z-v)dz.
%\end{equation}
For $T$ exponential with parameter $\paramT$, we have
\begin{equation}\label{sadfgrfhber}
f_{T}^{*n}(z-v)=\paramT\frac{(\paramT (z-v))^{n-1}}{(n-1)!}e^{-\paramT
(z-v)},\quad n=1,2,\dots.
\end{equation}
For $Y$ exponential with parameter $\paramY$, we have
\begin{equation}\label{saerfgewgh}
\P\big\{M(u+cz)=n\big\}=\frac{(\paramY(u+cz))^n}{n!}e^{-\paramY(u+cz)},\quad
n=1,2,\dots.
\end{equation}
%Put it in equation \eqref{qwretgjhmg}.

Bearing in mind that modified Bessel function of the first kind of order $1$ is
\begin{equation}\label{APPEND=AFormula=10}
\BesselI{1}(z)=\sum_{k=0}^{\infty}\frac{1}{k!\,(k+1)!}\left(\frac{z}{2}\right)^{2k+1}
=\sum_{n=1}^{\infty}\frac{1}{n!\,(n-1)!}\left(\frac{z}{2}\right)^{2n-1},
\end{equation}
we put \eqref{sadfgrfhber} and \eqref{saerfgewgh} in \eqref{qwretgjhmg}. We
have
\begin{multline*}
\int_{v}^{t}\frac{u+cv}{u+cz}e^{-\paramY(u+cz)}\sum_{n=1}^{\infty}\frac{(\paramY(u+cz))^n}{n!}
\paramT\frac{(\paramT (z-v))^{n-1}}{(n-1)!}e^{-\paramT (z-v)}dz
\\
=\paramY\paramT\int_{v}^{t}(u+cv)
\sum_{n=1}^{\infty}\frac{\paramY^{n-1}\paramT^{n-1}(u+cz)^{n-1}(z-v)^{n-1}}{n!(n-1)!}e^{-\paramY(u+cz)}e^{-\paramT
(z-v)}dz
\\
=\sqrt{\paramY\paramT c}\,(v+u/c)e^{-\paramY u}e^{-\paramY cv}
\\
\times\int_{0}^{t-v}\frac{\BesselI{1}(2\sqrt{\paramY\paramT
c(y+v+u/c)y})}{\sqrt{(y+v+u/c)y}} e^{-(\paramY c+\paramT)y}dy,
\end{multline*}
as required. In the last equation we made the change of variables: $z-v=y$.
\end{proof}

%%%%%%%%%%%%%%%%%%%%%%%%%%%%%%%%%%%%%%%%%%%%%%%%%%%%%%%%%%%%%%%%%%%%%%%%%%%%%%%%%%%%%%%%
\section{Main technicalities and auxiliary results}\label{wqedtfyjhtgfj}
%%%%%%%%%%%%%%%%%%%%%%%%%%%%%%%%%%%%%%%%%%%%%%%%%%%%%%%%%%%%%%%%%%%%%%%%%%%%%%%%%%%%%%%%

%%%%%%%%%%%%%%%%%%%%%%%%%%%%%%%%%%%%%%%%%%%%%%%%%%%%%%%%%%%%%%%%%%%%%%%%%%%%%%%%%%%%%%%%
\subsection{Non-uniform Berry-Esseen bounds in local CLT}\label{srdthrj}
%%%%%%%%%%%%%%%%%%%%%%%%%%%%%%%%%%%%%%%%%%%%%%%%%%%%%%%%%%%%%%%%%%%%%%%%%%%%%%%%%%%%%%%%

Let the random vectors $\xi_{i}$, $i=1,2,\dots$, assuming values in $\R^{m}$ be
i.i.d. with c.d.f. $P$, with zero mean and with identity covariance matrix $I$.
Put $S_n=\frac{1}{\sqrt{n}}\sum_{i=1}^{n}\xi_{i}$, $\P_{n}(A)=\P\{S_{n}\in
A\}$, $A\subset\R^{m}$, $\p_n(x)=\frac{\partial^m}{\partial x_1,\dots,\partial
x_m}\P\{S_n\leqslant x\}$, $x=(x_1,\dots,x_m)\in\R^{m}$.

The Berry-Esseen bounds in one-dimensional, as $m=1$, central limit theorem
(CLT) are well known. The following theorem follows from Theorem~11 in \S~2 of
\citeNP{[Petrov 1975]} proved for non-identically distributed random variables
$\xi_{i}$, $i=1,2,\dots$.
\begin{theorem}[\citeNP{[Petrov 1975]}]\label{qwretyuhjrk}
Let $\E\xi_1^2>0$, $\E|\xi_1|^3<\infty$, and
$\int_{|t|>\epsilon}|\E{e^{it\xi_1}}|^ndt=\underline{O}(n^{-1})$ for any fixed
$\epsilon>0$. Then for all sufficiently large $n$ a bounded p.d.f. $\p_n(x)$
exists and
\begin{equation*}
\sup_{x\in\R}\left|\,\p_n(x)-\Ugauss{0}{1}(x)\right|=\underline{O}(n^{-1/2}),\quad
n\to\infty.
\end{equation*}
\end{theorem}

The non-uniform Berry-Esseen bounds in integral rather than local
one-dimensional CLT
%central limit theorem
may be found in \citeNP{[Petrov 1995]} (see, e.g., Theorems 15 and 14 in Ch. 5,
\S~6 in \citeNP{[Petrov 1995]}).

A detailed study of normal approximations and asymptotic expansions in the CLT
%central limit theorem
in $\R^{m}$, as $m>1$, is conducted in \citeNP{[Bhattacharya Rao 1976]} (see
particularly Theorem 19.2 in \citeNP{[Bhattacharya Rao 1976]}. The non-uniform
Berry-Esseen bounds in $\R^{m}$, $m>1$, that is used in
Section~\ref{sdfyguklyi} as auxiliary result, is Theorem~4 in \S~3 of
\citeNP{[Dubinskaite 1982]} with $k=m$ and $s=2$. We first formulate the
following conditions.

\emph{Condition} ($P_m$): there exists $N\geqslant 1$ such that
$\sup_{x\in\R^m}p_N(x)\leqslant C<\infty$ and
\begin{equation*}
%\psi_{2,n}=
\int_{\|x\|>\sqrt{n}}\|x\|^{2}P(dx)+\frac{1}{n}\int_{\|x\|\leqslant\sqrt{n}}\|x\|^{4}P(dx)
+\frac{1}{\sqrt{n}}\sup_{\|e\|=1}\Big|\int_{\|x\|\leqslant\sqrt{n}}(x,e)^{3}P(dx)\Big|
=\underline{O}(\epsilon_n),
\end{equation*}
$n\to\infty$, where $\epsilon_n$ is a sequence of positive numbers such that
$\epsilon_n\to 0$, as $n\to\infty$, and $\epsilon_n\geqslant 1/\sqrt{n}$.

\emph{Condition} ($A_2$): $\beta_{2}=\E\|\xi_1\|^2<\infty$,
$\alpha_1(t)=\E(\xi_1,t)<\infty$.

\begin{theorem}[\citeNP{[Dubinskaite 1982]}]\label{w4e5y46yu43}
To have
\begin{equation}\label{ewyjhfyhm}
\sup_{x\in\R^{m}}(1+\|x\|)^3\left|\,\p_n(x)-\Ugauss{0}{I}(x)\right|=\underline{O}(n^{-1/2}),\quad
n\to\infty,
\end{equation}
it is necessary and sufficient that conditions $(P_m)$, $(A_2)$, and
\begin{equation*}
z\int_{\|x\|>z}\|x\|^{2}P(dx)+\sup_{\|e\|=1}\Big|\int_{\|x\|\leqslant
z}(x,e)^{3}P(dx)\Big|=\underline{O}(1),\quad z\to\infty,
\end{equation*}
be satisfied.
\end{theorem}

%\begin{remark}
%It is known that the estimate \eqref{ewyjhfyhm} is optimal in terms of
%dependence on $\|x\|$, i.e., the power $3$ in \eqref{ewyjhfyhm} can not be
%replaced by a greater one.
%\end{remark}

%\begin{theorem}
Under similar conditions, with modified ($P_m$), asymptotical expansions in
Theorem~\ref{w4e5y46yu43} is (see \citeNP{[Dubinskaite 1982]})
\begin{equation*}
\sup_{x\in\R^{m}}(1+\|x\|)^4\Big|\,\p_n(x)-\Big(\Ugauss{0}{I}(x)+n^{-1/2}
P_1(-\Ugauss{0}{I}:\{\chi_{\nu}\})(x)\Big)\Big|=\underline{O}(n^{-1}),\quad
n\to\infty,
\end{equation*}
%\end{theorem}
where (see Equation (7.20) in \citeNP{[Bhattacharya Rao 1976]})
\begin{multline*}
P_1(-\Ugauss{0}{I}:\{\chi_{\nu}\})(x)=\Big\{-\frac{1}{6}\Big[\chi_{(3,0,\dots,0)}(-x_1^3+3x_1)+\dots
+\chi_{(0,0,\dots,3)}(-x_m^3+3x_m)\Big]
\\
-\frac{1}{2}\Big[\chi_{(2,1,0,\dots,0)}(-x_{1}^2x_{2}+x_{2})+\dots
+\chi_{(0,\dots,0,1,2)}(-x_{m}^2x_{m-1}+3x_{m-1})\Big]
\\
-\Big[\chi_{(1,1,1,0,\dots,0)}(-x_{1}x_{2}x_{3}+x_{2})+\dots
+\chi_{(0,\dots,0,0,1,1)}(-x_{m}x_{m-1}x_{m-2})\Big]\Big\}\Ugauss{0}{I}(x)
\end{multline*}
for $x=(x_1,x_2,\dots,x_m)\in\R^{m}$. In the particular case when
$\chi_{(2,1,0,\dots,0)}=\dots=\chi_{(0,\dots,0,1,2)}=0$ and
$\chi_{(1,1,1,0,\dots,0)}=\dots=\chi_{(0,\dots,0,0,1,1)}=0$, we have
\begin{multline*}
P_1(-\Ugauss{0}{I}:\{\chi_{\nu}\})(x)=-\frac{1}{6}\Big[\chi_{(3,0,\dots,0)}(-x_1^3+3x_1)+\dots
+\chi_{(0,0,\dots,3)}(-x_m^3+3x_m)\Big]\Ugauss{0}{I}(x).
\end{multline*}
In the case $m=2$, for $x=(x_1,x_2)$ we have
%\begin{multline*}
%P_1(-\Ugauss{0}{I}:\{\chi_{\nu}\})(x)=-\frac{1}{6}\Big[\chi_{(3,0)}(-x_1^3+3x_1)+
%\chi_{(0,3)}(-x_2^3+3x_2)\Big]\Ugauss{0}{1}(x_1)\Ugauss{0}{1}(x_2),
%\end{multline*}
%or
\begin{multline*}
P_1(-\Ugauss{0}{I}:\{\chi_{\nu}\})(x)=\frac{1}{6}\Big[\chi_{(3,0)}(x_1^3-3x_1)+
\chi_{(0,3)}(x_2^3-3x_2)\Big]\Ugauss{0}{1}(x_1)\Ugauss{0}{1}(x_2).
\end{multline*}

%%%%%%%%%%%%%%%%%%%%%%%%%%%%%%%%%%%%%%%%%%%%%%%%%%%%%%%%%%%%%%%%%%%%%%%%%%%%%%%%%%%%%%%%
\subsection{Fundamental identities}\label{werwhteyjn}
%%%%%%%%%%%%%%%%%%%%%%%%%%%%%%%%%%%%%%%%%%%%%%%%%%%%%%%%%%%%%%%%%%%%%%%%%%%%%%%%%%%%%%%%

These identities were established and used in \citeNP{[Malinovskii 2017]}. For
$B_1=(\E{T})^2\D{Y}+(\E{Y})^2\D{T}$, $B_2=\E{Y}\D{T}$, $B_3=\E{T}\D{Y}$, and
$B_4=\D{Y}\D{T}$, we use notation
\begin{equation}\label{wqerertjr4}
\begin{aligned}
\mathcal{Y}_{n}(\OneVar)&=\frac{\OneVar-n\E{Y}}{\sqrt{n\D{Y}}},&\mathcal{T}_{n}(\TwoVar)&=\frac{\TwoVar-n\E{T}}{\sqrt{n\D{T}}},
\\
\Delta_{n}(\OneVar,\TwoVar)&=\dfrac{\TwoVar\E{Y}-\OneVar\E{T}}{\sqrt{B_1n}},&
\Lambda_{n}(\OneVar,\TwoVar)&=\dfrac{B_1n-(B_2\OneVar+B_3\TwoVar)}{\sqrt{B_1B_4n}}.
\end{aligned}
\end{equation}

\begin{lemma}\label{erwtjytk}
We have the identity
\begin{equation*}
\mathcal{Y}^2_{n}(\OneVar)+\mathcal{T}^2_{n}(\TwoVar)=\Delta^2_{n}(\OneVar,\TwoVar)+\Lambda_n^2(\OneVar,\TwoVar).
\end{equation*}
\end{lemma}

\begin{proof}
Getting of this identity is based on algebraic manipulations with the left-hand
side, aimed at completing the square. Its proof may be done as well by means of
a straightforward check.
\end{proof}

\begin{lemma}\label{asdffgfh}
We have the identity
\begin{equation*}
\Lambda_{n+1}(\OneVar,\TwoVar)-\Lambda_{n}(\OneVar,\TwoVar)=\left(\frac{B_1}{B_4n}\right)^{1/2}
+\Lambda_{n+1}(\OneVar,\TwoVar)\big(1-\sqrt{1+1/n}\big).
\end{equation*}
\end{lemma}

%\begin{proof}
%We have
%\begin{multline*} (B_1B_4n)^{1/2}[\Lambda_{n+1}(\OneVar,\TwoVar)-\Lambda_{n}(\OneVar,\TwoVar)]
%\\
%=\big\{(B_1B_4(n+1))^{1/2}\Lambda_{n+1}(\OneVar,\TwoVar)-(B_1B_4n)^{1/2}\Lambda_{n}(\OneVar,\TwoVar)\big\}
%\\
%+\big\{(B_1B_4n)^{1/2}\Lambda_{n+1}(\OneVar,\TwoVar)-(B_1B_4(n+1))^{1/2}\Lambda_{n+1}(\OneVar,\TwoVar)\big\}
%\\
%=B_1+\Lambda_{n+1}(\OneVar,\TwoVar)(B_1B_4n)^{1/2}\big(1-\sqrt{1+1/n}\big).
%\end{multline*}
%Indeed, since
%\begin{equation*}
%(B_1B_4(n+1))^{1/2}\Lambda_{n+1}(\OneVar,\TwoVar)=B_1(n+1)-(B_2\OneVar+B_3\TwoVar),
%\end{equation*}
%\begin{equation*}
%(B_1B_4n)^{1/2}\Lambda_{n}(\OneVar,\TwoVar)=B_1n-(B_2\OneVar+B_3\TwoVar),
%\end{equation*}
%the first summand is
%\begin{multline*}
%(B_1B_4(n+1))^{1/2}\Lambda_{n+1}(\OneVar,\TwoVar)-(B_1B_4n)^{1/2}\Lambda_{n}(\OneVar,\TwoVar)
%\\
%=B_1(n+1)-(B_2\OneVar+B_3\TwoVar)-B_1n+(B_2\OneVar+B_3\TwoVar)=B_1.
%\end{multline*}
%The second summand is
%\begin{multline*}
%(B_1B_4n)^{1/2}\Lambda_{n+1}(\OneVar,\TwoVar)-(B_1B_4(n+1))^{1/2}\Lambda_{n+1}(\OneVar,\TwoVar)
%\\[4pt]
%=(B_1B_4n)^{1/2}\Lambda_{n+1}(\OneVar,\TwoVar)\bigg\{1-\frac{(B_1B_4(n+1))^{1/2}}{(B_1B_4n)^{1/2}}\bigg\}
%\\[2pt]
%=(B_1B_4n)^{1/2}\Lambda_{n+1}(\OneVar,\TwoVar)\bigg\{1-\sqrt{\frac{(n+1)}{n}}\,\bigg\}.
%\end{multline*}
%The proof is complete.
%\end{proof}

\begin{lemma}\label{t7i7ityo7oi}
We have the identities
\begin{equation*}
1-\frac{\OneVar}{n\E{Y}}=\frac{\sqrt{B_4}}{\sqrt{B_1n}}\Lambda_{n}(\OneVar,\TwoVar)
+\frac{B_3}{\E{Y}\sqrt{B_1n}}\Delta_{n}(\OneVar,\TwoVar)
\end{equation*}
and
\begin{equation*}
1-\sqrt{\frac{\OneVar}{n\E{Y}}}=\bigg\{\frac{\sqrt{B_4}}{\sqrt{B_1n}}\Lambda_{n}(\OneVar,\TwoVar)
+\frac{B_3}{\E{Y}\sqrt{B_1n}}\Delta_{n}(\OneVar,\TwoVar)\bigg\}\bigg(1+\sqrt{\frac{\OneVar}{n\E{Y}}}\bigg).
\end{equation*}
\end{lemma}

%\begin{proof}
%Bearing in mind that $B_1/\E{Y}-B_2=(\E{T})^2\D{Y}/\E{Y}$, we have
%\begin{multline*}
%\Lambda_{n}(\OneVar,\TwoVar)=\left(1-\frac{\OneVar}{\E{Y}n}\right)\left(\dfrac{B_1n}{B_4}\right)^{1/2}+\frac{\OneVar
%B_1/\E{Y}}{\sqrt{B_1B_4n}} -\dfrac{B_2\OneVar+B_3\TwoVar}{\sqrt{B_1B_4n}}
%\\
%=\left(1-\frac{\OneVar}{\E{Y}n}\right)\left(\dfrac{B_1n}{B_4}\right)^{1/2}
%-\frac{\E{T}\D{Y}}{\E{Y}\sqrt{B_4}}\left(\frac{\TwoVar\E{Y}-\OneVar\E{T}}{\sqrt{B_1n}}\right)
%\\
%=\left(1-\frac{\OneVar}{\E{Y}n}\right)\left(\dfrac{B_1n}{B_4}\right)^{1/2}
%-\frac{\E{T}\D{Y}}{\E{Y}\sqrt{B_4}}\,\Delta_{n}(\OneVar,\TwoVar).
%\end{multline*}
%Rewrite it
%\begin{multline*}
%1-\frac{\OneVar}{n\E{Y}}=\left(\dfrac{B_4}{B_1n}\right)^{1/2}
%\left(\Lambda_{n}(\OneVar,\TwoVar)+\frac{\E{T}\D{Y}}{\E{Y}\sqrt{B_4}}\Delta_{n}(\OneVar,\TwoVar)\right)
%\\
%=\left(\dfrac{1}{B_1n}\right)^{1/2}\left(\sqrt{B_4}\Lambda_{n}(\OneVar,\TwoVar)+\frac{B_3}{\E{Y}}\Delta_{n}(\OneVar,\TwoVar)\right),
%\end{multline*}
%as required. The proof is complete.
%\end{proof}

\begin{lemma}\label{asdfhjnmkgh}
We have the identity
\begin{equation*}
\begin{aligned}
1-\frac{\TwoVar}{n\E{T}}&=\frac{1}{\sqrt{B_1}\sqrt{n}}\bigg(\sqrt{B_4}\Lambda_{n}(\OneVar,\TwoVar)-\frac{B_2}{\E{T}}\Delta_{n}(\OneVar,\TwoVar)
\bigg).
\end{aligned}
\end{equation*}
\end{lemma}

\begin{remark}
The identities of Lemmas \ref{erwtjytk}--\ref{asdfhjnmkgh} in a more general
form were proved and used first in \citeNP{[Malinovskii 1993]}.
\end{remark}

%\begin{equation*}
%B_1=(\E{T})^2\D{Y}+(\E{Y})^2\D{T},\ B_2=\E{Y}\D{T},\ B_3=\E{T}\D{Y},\
%B_4=\D{Y}\D{T},
%\end{equation*}

\begin{lemma}
We have the identities
\begin{equation*}
\begin{aligned}
\mathcal{Y}_{n}(\OneVar)&=-\frac{1}{n^{2}\E{Y}\sqrt{\D{Y}B_1}}
\bigg(\sqrt{B_4}\Lambda_{n}(\OneVar,\TwoVar)
+\frac{B_3}{\E{Y}}\Delta_{n}(\OneVar,\TwoVar)\bigg),
\\[6pt]
\mathcal{T}_{n}(\TwoVar)&=-\frac{1}{n^2\E{T}\sqrt{\D{T}B_1}}
\bigg(\sqrt{B_4}\Lambda_{n}(\OneVar,\TwoVar)-\frac{B_2}{\E{T}}\Delta_{n}(\OneVar,\TwoVar)
\bigg).
\end{aligned}
\end{equation*}
\end{lemma}

\subsection{Approximation of integral sum by the corresponding integral}\label{ertklyuty}
%%%%%%%%%%%%%%%%%%%%%%%%%%%%%%%%%%%%%%%%%%%%%%%%%%%%%%%%%%%%%%%%%%%%%%%%%%%%%%%%%%%%%%%%

\begin{lemma}\label{qaerfgwehg}
Let the function $f$ be differentiable sufficient number of times. We have
\begin{equation*}
\sum_{i=1}^{\nu-1}f(\xi_{i})(\xi_{i+1}-\xi_{i})=\int_{\xi_{1}}^{\xi_{\nu}}f(z)dz
-\frac{1}{2}\sum_{i=1}^{\nu-1}f^{\prime}(\xi_{i})(\xi_{i+1}-\xi_{i})^2
%\\
-\frac{1}{6}\sum_{i=1}^{\nu-1}f^{\prime\prime}(\xi_{i})(\xi_{i+1}-\xi_{i})^3+\dots.
\end{equation*}
\end{lemma}

\begin{proof}
By Taylor's formula, we have
\begin{equation*}
f(z)=f(\xi_{i})+f^{\prime}(\xi_{i})(z-\xi_{i})+\frac12f^{\prime\prime}(\xi_{i})(z-\xi_{i})^2+\dots.
\end{equation*}
Integrating it, we have
\begin{multline*}
\int_{\xi_{i}}^{\xi_{i+1}}f(z)dz
=f(\xi_{i})\underbrace{\int_{\xi_{i}}^{\xi_{i+1}}dz}_{(\xi_{i+1}-\xi_{i})}+f^{\prime}(\xi_{i})
\underbrace{\int_{\xi_{i}}^{\xi_{i+1}}(z-\xi_{i})dz}_{\frac12(\xi_{i+1}-\xi_{i})^2}
+\frac{1}{2}f^{\prime\prime}(\xi_{i})
\underbrace{\int_{\xi_{i}}^{\xi_{i+1}}(z-\xi_{i})^2dz}_{\frac{1}{3}(\xi_{i+1}-\xi_{i})^3}+\dots
\\
=f(\xi_{i})(\xi_{i+1}-\xi_{i})+\frac{1}{2}f^{\prime}(\xi_{i})(\xi_{i+1}-\xi_{i})^2
+\frac{1}{6}f^{\prime\prime}(\xi_{i})(\xi_{i+1}-\xi_{i})^3+\dots.
\end{multline*}
Finally, we obtain
\begin{multline*}
\int_{\xi_{1}}^{\xi_{n}}f(z)dz=\sum_{i=1}^{\nu-1}\int_{\xi_{i}}^{\xi_{i+1}}f(z)dz
=\sum_{i=1}^{\nu-1}f(\xi_{i})(\xi_{i+1}-\xi_{i})
\\
+\frac{1}{2}\sum_{i=1}^{\nu-1}f^{\prime}(\xi_{i})(\xi_{i+1}-\xi_{i})^2
+\frac{1}{6}\sum_{i=1}^{\nu-1}f^{\prime\prime}(\xi_{i})(\xi_{i+1}-\xi_{i})^3+\dots,
\end{multline*}
which is required.
\end{proof}

%%%%%%%%%%%%%%%%%%%%%%%%%%%%%%%%%%%%%%%%%%%%%%%%%%%%%%%%%%%%%%%%%%%%%%%%%%%%%%%%%%%%%%%%
\subsection{Modified Bessel function of the second kind}\label{sdrthjrnm}
%%%%%%%%%%%%%%%%%%%%%%%%%%%%%%%%%%%%%%%%%%%%%%%%%%%%%%%%%%%%%%%%%%%%%%%%%%%%%%%%%%%%%%%%

Modified Bessel function of the second kind of order $\nu$, or Macdonald
function\footnote{There exists a significant discrepancy in names of these
functions. For example, quoting \citeNP{[Magnus Oberhettinger 1953]}, we see
``modified Hankel function'' (p.~3), ``modified Bessel function of the third
kind or Basset's function (although the present definition is due to
Macdonald)'' (p.~5).}, is defined in \citeNP{[Magnus Oberhettinger 1953]} as
\begin{equation*}
\BesselK{\nu}(z)=\frac{\pi}{2\sin(\nu\pi)}[\BesselI{-\nu}(z)-\BesselI{\nu}(z)],
\end{equation*}
where $\BesselI{\nu}(z)$ is the modified Bessel function of the first kind of
order $\nu$. It immediately follows that $\BesselK{\nu}(z)=\BesselK{-\nu}(z)$.

It has an integral representation (see, e.g., \citeNP{[Gradshtein 1980]},
formula 8.432 (6))
\begin{equation*}
\BesselK{\nu}(z)=\frac{1}{2}\Big(\frac{z}{2}\Big)^{\nu}\int^{\infty}_{0}\frac{1}{t^{\nu+1}}
\exp\bigg\{-\Big(t+\frac{z^2}{4t}\Big)\bigg\}\,dt.
\end{equation*}
This equation is checked, e.g., in \citeNP{[Glasser 2012]}.

It is well known (see, e.g., \citeNP{[Gradshtein 1980]}, formula 8.432 (7) and
\citeNP{[Magnus Oberhettinger 1953]}, \S~7.12 formula (23) on p.~82) that for
$x>0$ and $z>0$
\begin{equation}\label{wergetrh}
\BesselK{\nu}(xz)=\frac{z^{\nu}}{2}\int_{0}^{\infty}\frac{1}{t^{\nu+1}}
\exp\Bigg\{-\frac{x}{2}\bigg(t+\frac{z^2}{t}\bigg)\Bigg\}\,dt.
\end{equation}
In particular, for $x>0$, $z>0$, and $\nu=\frac{1}{2}$, we have
\begin{equation}\label{ertgerhr}
\BesselK{\frac12}(xz)\eqOK\frac{z^{1/2}}{2}\int_{0}^{\infty}\frac{1}{t^{3/2}}
\exp\Bigg\{-\frac{x}{2}\bigg(t+\frac{z^2}{t}\bigg)\Bigg\}\,dt
\eqOK\sqrt{\frac{\pi}{2zx}}\exp\big\{-xz\big\}.
\end{equation}

%The following asymptotic expansions for the modified Bessel functions were
%developed by H.\,Hankel.

\begin{lemma}[Hankel's power series expansion]\label{sdrjrnm}
We have
\begin{multline*}
\BesselK{\alpha}(z)=\frac{\sqrt{\pi}}{\sqrt{2z}}e^{-z}\Bigg(1+\frac{4\alpha^2-1}{8z}
+\frac{(4\alpha^2-1)(4\alpha^2-9)}{2!(8z)^2}
\\
+\frac{(4\alpha^2-1)(4\alpha^2-9)(4\alpha^2-25)}{3!(8z)^3}+\dots\Bigg).
\end{multline*}
\end{lemma}

They yield $\BesselK{1/2}(x)=\frac{\sqrt{\pi}}{\sqrt{2x}}e^{-x}$ (cf.
\eqref{ertgerhr}),
$\BesselK{3/2}(z)\eqOK\frac{\sqrt{\pi}}{\sqrt{2z}}e^{-z}(1+z^{-1})$,
$\BesselK{5/2}(z)\eqOK\frac{\sqrt{\pi}}{\sqrt{2z}}e^{-z}(1+3z^{-1}+3z^{-2})$,
and so on.

%%%%%%%%%%%%%%%%%%%%%%%%%%%%%%%%%%%%%%%%%%%%%%%%%%%%%%%%%%%%%%%%%%%%%%%%%%%%%%%%%%%%%%%%
\subsection{Incomplete modified Bessel function and inverse Gaussian distribution}\label{sdgtjhrtjrtf}
%%%%%%%%%%%%%%%%%%%%%%%%%%%%%%%%%%%%%%%%%%%%%%%%%%%%%%%%%%%%%%%%%%%%%%%%%%%%%%%%%%%%%%%%

Considering the integral in \eqref{ertgerhr} with arbitrary limits of
integration $0<\A<\B$, we introduce incomplete modified Bessel function of the
second kind of order $\frac12$
\begin{equation}\label{ewthtrjt}
\BesselK{\frac12}(xz\mid
\A,\B)=\frac{z^{1/2}}{2}\int_{\A}^{\B}\frac{1}{t^{3/2}}
\exp\Bigg\{-\frac{x}{2}\bigg(t+\frac{z^2}{t}\bigg)\Bigg\}\,dt.
\end{equation}

It is noteworthy that the integral in \eqref{ewthtrjt} is the same in the
expression for c.d.f. of inverse Gaussian distribution \eqref{wqrdtgrehr} with
shape parameter $\lambdaIG>0$ and mean parameter $\muIG>0$, and in the
expression for the incomplete modified Bessel function of the second kind of
order $\frac12$ introduced above. Indeed, we easily have
\begin{equation}\label{qwertgehrt}
\begin{aligned}
\int_{0}^{x}\frac{1}{t^{3/2}}
\exp\bigg\{-\frac{\lambdaIG}{2\muIG^2}\bigg(t+\frac{\muIG^2}{t}\bigg)\bigg\}dt
&=\exp\bigg\{-\frac{\lambdaIG}{\muIG}\bigg\}\Big(\frac{2\pi}{\lambdaIG}\Big)^{1/2}F\big(x;\muIG,\lambdaIG,-\tfrac{1}{2}\big)
%F(x;\muIG,\lambdaIG)
\\
&=\frac{2}{\muIG^{1/2}}\BesselK{\frac12}\Big(\frac{\lambdaIG}{\muIG^2}\muIG\,\Big|\,0,x\Big),
\end{aligned}
\end{equation}
and
\begin{equation*}
\BesselK{\frac12}\Big(\frac{\lambdaIG}{\muIG^2}\muIG\,\Big|\,0,x\Big)
=\exp\bigg\{-\frac{\lambdaIG}{\muIG}\bigg\}\Big(\frac{\pi\muIG}{2\lambdaIG}\Big)^{1/2}F\big(x;\muIG,\lambdaIG,-\tfrac{1}{2}\big),
%F(x;\muIG,\lambdaIG),
\end{equation*}
or vice versa
\begin{equation*}
F\big(x;\muIG,\lambdaIG,-\tfrac{1}{2}\big)
%F(x;\muIG,\lambdaIG)
=\Big(\frac{2\lambdaIG}{\pi\muIG}\Big)^{1/2}\exp\bigg\{\frac{\lambdaIG}{\muIG}\bigg\}
\BesselK{\frac12}\Big(\frac{\lambdaIG}{\muIG^2}\muIG\,\Big|\,0,x\Big).
\end{equation*}

Making the change of variables $t=\tau^2$, $\tau=\sqrt{t}$, $2\tau d\tau=dt$ in
this integral, we have
\begin{equation*}
2\int_{0}^{\sqrt{x}}\frac{1}{\tau^2}
\exp\bigg\{-\frac{\lambdaIG}{2\muIG^2}\bigg(\tau^2+\frac{\muIG^2}{\tau^2}\bigg)\bigg\}\,d\tau.
\end{equation*}
Such integrals were studied in \citeNP{[Binet 1841]}.

%%%%%%%%%%%%%%%%%%%%%%%%%%%%%%%%%%%%%%%%%%%%%%%%%%%%%%%%%%%%%%%%%%%%%%%%%%%%%%%%%%%%%%%%
\subsection{Binet's integrals}\label{rtgwerghwer}
%%%%%%%%%%%%%%%%%%%%%%%%%%%%%%%%%%%%%%%%%%%%%%%%%%%%%%%%%%%%%%%%%%%%%%%%%%%%%%%%%%%%%%%%

\begin{theorem}[\citeNP{[Binet 1841]}]\label{wqertheyjre}
For $0<\A<\B$, we have
\begin{multline}\label{dtfghjftgj}
\int_{\sqrt{\A}}^{\sqrt{\B}}\frac{1}{\tau^2}
\exp\bigg\{-\frac{x}{2}\bigg(\tau^2+\frac{z^2}{\tau^2}\bigg)\bigg\}\,d\tau
\eqOK
2e^{xz}z^{-1/2}\int_{\sqrt{\A}+\frac{z}{\sqrt{\A}}}^{\sqrt{\B}+\frac{z}{\sqrt{\B}}}
\frac{e^{-\frac{x}{2}\tau^2}}{(\tau+\sqrt{\tau^2-4z})^2}d\tau
\\
+2e^{-xz}z^{-1/2}\int_{\sqrt{\A}-\frac{z}{\sqrt{\A}}}^{\sqrt{\B}-\frac{z}{\sqrt{\B}}}
\frac{e^{-\frac{x}{2}\tau^2}}{(\sqrt{\tau^2+4z}+\tau)^2}d\tau.
\end{multline}
\end{theorem}

Bearing in mind two standard indefinite integrals
\begin{multline*}
\int\frac{e^{-qu^2}}{(u+\sqrt{u^2-4r})^2}du=-\frac{e^{-qu^2}u}{16qr^2}
+\frac{e^{-4qr-q(-4r+u^2)}\sqrt{-4r+u^2}}{16qr^2}
\\
+\frac{\sqrt{\pi}(2\UGauss{0}{1}(\sqrt{2q}u)-1)}{32q^{3/2}r^2}
-\frac{\sqrt{\pi}(2\UGauss{0}{1}(\sqrt{2q}u)-1)}{8\sqrt{q}r}
\\
-\frac{e^{-4qr}\sqrt{\pi} (2\UGauss{0}{1}(\sqrt{2q}\sqrt{-4r+u^2})-1)
}{32q^{3/2}r^2},
\end{multline*}
%and
\begin{multline*}
\int\frac{e^{-qx^2}}{(\sqrt{x^2+4r}+x)^2}dx=-\frac{e^{-q x^2} x}{16 q
r^2}+\frac{e^{4qr-q(4r+x^2)}\sqrt{4r+x^2}}{16qr^2}
\\
+\frac{\sqrt{\pi}(2\UGauss{0}{1}(\sqrt{2q}x)-1)}{32q^{3/2}r^2}
+\frac{\sqrt{\pi}(2\UGauss{0}{1}(\sqrt{2q}x)-1)}{8\sqrt{q}r}
\\
-\frac{e^{4qr}\sqrt{\pi}(2\UGauss{0}{1}(\sqrt{2q}\sqrt{4r+x^2})-1)}{32q^{3/2}r^2},
\end{multline*}
which may be verified by direct differentiation, we express \eqref{dtfghjftgj}
in terms of c.d.f. of a standard normal distribution.

The proof of Theorem~\ref{wqertheyjre} is based on Binet's observation that,
\begin{multline*}
\int_{A}^{B}\frac{1}{y^2}e^{-q(y^2+r^2y^{-2})}dy\eqOK2e^{2qr}\int_{A+rA^{-1}}^{B+rB^{-1}}
\frac{e^{-qu^2}}{(u+\sqrt{u^2-4r})^2}du
\\
+2e^{-2qr}\int_{A-rA^{-1}}^{B-rB^{-1}}\frac{e^{-qx^2}}{(\sqrt{x^2+4r}+x)^2}dx.
\end{multline*}
Indeed, writing $y=\frac{1}{2}(u+x)$, where $u=y+ry^{-1}$ and $x=y-ry^{-1}$,
one has
\begin{equation*}
y^2+r^2y^{-2}\eqOK x^2+2r\eqOK u^2-2r,
\end{equation*}
which yields $e^{-q(y^2+r^2y^{-2})}\eqOK e^{-q(x^2+2r)}\eqOK e^{-q(u^2-2r)}$.
Since $dy=\frac{1}{2}dx+\frac{1}{2}du$ and
\begin{equation*}
\frac{1}{y^2}=4(u+x)^{-2}=4(u+\sqrt{u^2-4r})^{-2}=4(\sqrt{x^2+4r}+x)^{-2},
\end{equation*}
the proof becomes obvious.

\end{document}